\documentclass[a4paper,12pt,reqno]{amsart}
\usepackage{vmargin}
\usepackage[pdftex]{graphicx}
\usepackage{amsfonts}
\usepackage{amssymb}
\usepackage[sorted]{amsrefs} 
\usepackage{mathrsfs}
\usepackage{stmaryrd}
\usepackage{amsmath}
\usepackage{amsthm}
\usepackage[shortlabels]{enumitem}
\usepackage{nomencl}
\usepackage[bookmarks=true]{hyperref}   
\usepackage{esint}
\usepackage{dsfont}
\usepackage{upgreek}
\usepackage{xcolor}
\usepackage{slashed}
\usepackage{scalerel}
\usepackage{comment}
\usepackage{todonotes}
\usepackage[utf8]{inputenc}

\usepackage{color} 
      
        \definecolor{pink}{rgb}{1,0,1}

        \newcommand{\Vol}{\operatorname{Vol}}

        \definecolor{purple}{rgb}{0.4,0.2,1}

\hypersetup{
    colorlinks,%
}

\author{Lashi Bandara}
\author{Medet Nursultanov}
\author{Julie Rowlett} 
\title[Weyl's law for weighted Laplace equation on RRM with boundary]{Eigenvalue asymptotics for weighted Laplace equations on rough Riemannian manifolds with boundary}
\date{\today}
\address{Lashi Bandara, 
Institut f\"ur Mathematik,
Universit\"at Potsdam, 
D-14476, Potsdam OT Golm, Germany
}
\urladdr{\href{http://www.math.uni-potsdam.de/~bandara}{http://www.math.uni-potsdam.de/~bandara}}
\email{\href{mailto:lashi.bandara@uni-potsdam.de}{lashi.bandara@uni-potsdam.de}}

\address{Medet Nursultanov, 
Mathematical Sciences,
Chalmers University of Technology and University of Gothenburg, 
SE-412 96, Gothenburg, Sweden}
\email{\href{mailto:medet@chalmers.se}{medet@chalmers.se}}

\address{Julie Rowlett,  Mathematical Sciences, Chalmers University of Technology and University of Gothenburg,  SE-412 96, Gothenburg, Sweden} 
\urladdr{\href{http://www.math.chalmers.se/~rowlett}{http://www.math.chalmers.se/~rowlett}}
\email{\href{mailto:julie.rowlett@chalmers.se}{julie.rowlett@chalmers.se}}

\dedicatory{This work is dedicated to the memory of Alan McIntosh.}


\def\colour{\colour}

\newcommand{\dbrac}[1]{\left\{#1\right\}}

\newcommand{\modulus}[1]{|#1|}

\newcommand{\norm}[1]{\| #1 \|}			
\newcommand{\set}[1]{\dbrac{#1}}

\newcommand{\intersect}{\cap}

\newcommand{\embed}{\hookrightarrow}		

\newcommand{\cW}{\mathcal{W}}
\newcommand{\cV}{\mathcal{V}}
\newcommand{\graph}{\mathrm{graph}}

\DeclareMathOperator{\spt}{spt}
\newcommand{\rest}[1]{{{\lvert_{}}_{}}_{#1}}

\newcommand{\pa}{\partial}

\mathchardef\semic="303B

\newcommand{\R}{{\mathbf R}}
\newcommand{\C}{{\mathbf C}}

\newcommand{\mH}{{\mathcal H}}

\newcommand{\dom}{\textsf{D}}

\newcommand{\barint}{\mbox{$ave \int$}}
\newcommand{\divv}{{\text{{\rm div}}}}

\newcommand{\cL}{\mathcal L}

\def\barint_#1{\mathchoice
            {\mathop{\vrule width 6pt
height 3 pt depth -2.5pt
                    \kern -8.8pt
\intop}\nolimits_{#1}}%
            {\mathop{\vrule width 5pt height
3 pt depth -2.6pt
                    \kern -6.5pt
\intop}\nolimits_{#1}}%
            {\mathop{\vrule width 5pt height
3 pt depth -2.6pt
                    \kern -6pt
\intop}\nolimits_{#1}}%
            {\mathop{\vrule width 5pt height
3 pt depth -2.6pt
          \kern -6pt \intop}\nolimits_{#1}}}

\definecolor{gr}{rgb}   {0.,   0.8,   0. }
\definecolor{bl}{rgb}   {0.,   0.5,   1. }
\definecolor{mg}{rgb}   {0.7,  0.,    0.7}

\newcommand{\Bk}{\color{black}}

\renewcommand{\emptyset}{\varnothing}

 \newtheorem{theorem}{Theorem}[section]
 \newtheorem{lemma}[theorem]{Lemma}
 \newtheorem{corollary}[theorem]{Corollary}

 \theoremstyle{definition}
 \newtheorem{definition}[theorem]{Definition}
 \newtheorem{example}[theorem]{Example}
 \newtheorem{remark}[theorem]{Remark}

\newtheorem{thm}{Theorem}[section]

\newtheorem{cor}[thm]{Corollary}
\newtheorem{prop}[thm]{Proposition}

\theoremstyle{definition}

\newtheorem{rem}[thm]{Remark}

\begin{document}

\vspace*{-2em}
\maketitle

\begin{abstract}
Our topological setting is a smooth compact manifold of dimension two or higher with smooth boundary.  Although this underlying topological structure is smooth, the Riemannian metric tensor is only assumed to be bounded and measurable.  This is known as a \em rough Riemannian manifold.  \em  For a large class of boundary conditions we demonstrate a Weyl law for the asymptotics of the eigenvalues of the Laplacian associated to a rough metric.  Moreover, we obtain eigenvalue asymptotics for weighted Laplace equations associated to a rough metric.  Of particular novelty is that the weight function is not assumed to be of fixed sign, and thus the eigenvalues may be both positive and negative.  Key ingredients in the proofs were demonstrated by Birman and  Solomjak nearly fifty years ago in their seminal work on eigenvalue asymptotics.  In addition to determining the eigenvalue asymptotics in the rough Riemannian manifold setting for weighted Laplace equations, we also wish to promote their achievements which may have further applications to modern problems.
\end{abstract}

\tableofcontents
\vspace*{-2em}

\parindent0cm
\setlength{\parskip}{\baselineskip}


\section{Introduction}  
Let $M$ be a smooth, $n$-dimensional topological manifold with smooth boundary, $\pa M$, such that the closure, $\overline{M} = M \cup \pa M$, is compact.  If $M$ is equipped with a smooth Riemannian metric, $g$, then there is a naturally associated Laplace operator, which in local coordinates is 
\begin{equation} \label{Deltag} \Delta_g = - \frac{1}{\sqrt{ \det(g)}} \sum_{i,j=1}^n \frac{\pa}{\pa x_i} \left( g^{ij} \sqrt{\det(g)} \frac{\pa}{\pa x_j}\ \cdot\ \right). \end{equation}
This is a second order elliptic operator with smooth coefficients, inherited from the smoothness of the Riemannian metric.  It is well known in this setting that the Laplacian, $\Delta_g$,  has a discrete, non-negative set of eigenvalues which accumulate only at $\infty$.  The set of all eigenvalues is known as the spectrum.  Connections between the spectrum and the geometry of the underlying manifold have captivated mathematicians and physicists alike for many years; see for example \cite{Weyl}, \cite{kac}, \cite{ms}, \cite{gww},\cite{hawking}, \cite{polyakov},  \cite{sos}, and \cite{corners}.  Whereas it is impossible in general to analytically compute the individual eigenvalues, in order to discover relationships between the eigenvalues and the geometry, one may study quantities determined by the spectrum.  Any such quantity is known as a spectral invariant.  The most fundamental spectral invariants are determined by the rate at which the eigenvalues tend to infinity and were discovered by Hermann Weyl in 1911 \cite{Weyl}.  

Weyl proved that in the special case in which $M = \Omega$ is a smoothly bounded domain in $\R^n$, and the Dirichlet boundary condition is taken for the Euclidean Laplacian, then 
\begin{equation} \label{weyllaw0} \lim_{\Lambda \to \infty} \frac{N(\Lambda) }{\Lambda^{n/2}} = \frac{\omega_n \Vol(\Omega)}{(2\pi)^n}. \end{equation}  
Above, $N(\Lambda)$ is the number of eigenvalues of the Laplacian, counted with multiplicity, which do not exceed $\Lambda$, $\omega_n$ is the volume of the unit ball in $\R^n$, and $\Vol(\Omega)$ is the volume of the domain, $\Omega$.  Hence, the rate at which the eigenvalues tend to infinity determines both the dimension, $n$, as well as the volume of $\Omega$.  These quantities are therefore spectral invariants.  The asymptotic formula \eqref{weyllaw0} is known as \em Weyl's Law.  \em  Weyl's law has both geometric generalisations, in which the underlying domain or manifold is no longer smooth; as well as analytic generalisations, in which the Laplace operator is replaced by a different, but typically Laplace-like operator.

Here we simultaneously consider both a geometric generalisation as well as an analytic generalisation.  We consider compact manifolds with a smooth differentiable structure and allow the possibility that such manifolds also carry a smooth boundary. However, the Riemannian-like metric in our setting, known as a \em rough metric, \em  is only assumed to be measurable, which is the primary novelty and a great source of difficulty in the analysis.  Such a \emph{rough metric} is only required to be bounded above in an $L^{\infty}$ sense, and essentially bounded below.  A smooth topological manifold, $M$, equipped with a rough Riemannian metric, $g$, is henceforth dubbed a \emph{rough Riemannian manifold,} with abbreviation RRM.  

Given that the coefficients of the metric tensor are merely measurable for a rough Riemannian manifold, the length functional over a curve is not well defined. Therefore, unlike for smooth or even continuous metrics, it is not possible to obtain a length structure via minimisation over curves. An alternative may be to consider supremums over the difference of certain classes of functions evaluated at two points in an attempt to obtain a distance between these points.  In the smooth context, locally Lipschitz functions with gradient almost-everywhere bounded above recovers the usual distance metric. In our context, it is not clear that this yields a reasonable notion of distance.  Rough metrics may have a dense set of singularities.  Moreover, given that we allow for boundary further complicates matters.  Although there is no canonical distance metric, we show that a rough metric does induce a canonical Radon measure which allows for an $L^p$ theory of tensors.  It is this fact, along with the fact that the exterior derivative is purely determined by the differential topology, that we will employ to see the rough metric as a measure space with a Dirichlet form that we will use to define a Laplacian.
 
The explicit study of these rough metrics arose with connections to the Kato square root problem (c.f. \cite{AHLMcT, AKMc, BMc}), where these metrics can be seen as geometric invariances of this problem (c.f. \cite{BRough}).  These objects have appeared implicitly in the past, particularly in the setting of real-variable harmonic analysis where the $L^\infty$ topology is a natural one.  They are also a useful device when singular information can be transferred purely into the Riemannian metric. This happens when the singular object is actually a differentiable manifold, and the singular information can be purely seen as a lack of regularity of the Riemannian metric.  A typical situation for this is when the manifold is obtained as a limit of a Riemannian manifolds in the Gromov-Hausdorff sense.

In the manifold context, rough metrics were treated in \cite{SC} by Saloff-Coste in connection with Harnack estimates.  \Bk Norris studied Lipschitz manifolds \cite{Norris}, where rough metrics  are the natural replacement of smooth Riemannian metrics due to the regularity of the differentiable structure.  Higher regularity versions of rough metrics, namely $C^0$ metrics, were used by Simon in \cite{Simon} to study the Ricci flow with initial data given by a $C^0$ metric.  Burtscher \cite{Burtscher} also used these higher regularity versions of rough metrics to study length structures, since for these metrics length structures exist as they do in the smooth context.

One of the main reasons to study general rough metrics with only bounded, measurable coefficients is that a pullback of a smooth metric by a lipeomorphism is only guaranteed to have such regularity.  Such a  transformation allows for objects with singularities to be studied more simply.  For example, a Euclidean box can be written as the Lipschitz graph over a sphere, and hence, a Euclidean box can be analysed as a rough metric on the sphere.  In \cite{BLM}, rough metrics played a central role in the analysis of the regularity properties  of a geometric flow that is related to the Ricci flow in the context of optimal transport.

Although rough metrics arise in a variety of contexts and have been studied by several authors, Weyl's law has remained unknown in this context.  Indeed, due to the highly singular nature and lack of a distance metric, one could expect results in the spirit of those for sub-Riemannian manifolds.  Although one may define a sub-Riemannian Laplacian which has discrete spectrum, spectral asymptotics are still a largely open question \cite{arous}, \cite{leandre}.  Under certain assumptions one does, however, have a Weyl asymptotic \cite{cdvht}.  Yet, in the same work, it is shown that in general there may be only a local Weyl asymptotic which varies from point to point; there is no single asymptotic rate at which the eigenvalues tend to infinity, thus no Weyl law for the eigenvalues of the Laplacian.  Given the lack of smoothness and the lack of a distance-metric structure in the rough Riemannian manifold context, it is not immediately clear whether or not one would expect a Weyl asymptotic for the eigenvalues of the Laplacian.  
 However, a very crude indication that this may be the case can be seen from the fact that the Kato square root problem can be solved for a rough metric on any closed manifold (c.f. \cite{BCont}). 


The Laplace operator for a rough metric in our analysis is applicable to a wide class of boundary conditions, which we call \emph{admissible boundary conditions.}  The precise definition of admissible boundary condition and examples are given in \S \ref{Sec:AnalPrel}.  In essence, we begin with a closed subspace $\cW$ of the Sobolev space $H^1(M)$ containing $H^1_0(M)$. Then, we  define a Dirichlet form on the subspace $\cW$, which in turn gives rise to an associated Laplace operator,  denoted $\Delta_{g, \cW}$.  We not only demonstrate Weyl's law for such a Laplace operator, but we also demonstrate a Weyl law for a weighted Laplace equation.  

To state our main result in full generality, let $M$ be a smooth compact manifold of dimension $n\geq 2$ with smooth boundary, and let $g$ be a rough metric on $M$ (see Definition \ref{def:roughmetric}).  Let $\beta>\frac{n}{2}$, and $\rho\in L^{\beta}(M, \ d\mu_g)$ be a real-valued function such that
\begin{equation*}
\int_{M}\rho \ d\mu_g\neq 0.
\end{equation*}
For an admissible boundary condition, $\cW$, we consider the form
\begin{equation*}
\mathcal{E}_{g,\cW}[u,v]=\left(\nabla u,\nabla v\right)_{L^2(M,\ d\mu_g)} 
\end{equation*}
defined on the subspace, 
\begin{equation*}
Z(\rho)=
\begin{cases}
\cW & \text{if }\mathcal{E}_{g,\cW} \text{ generates the norm in } \cW,\\
    & \text{which is equivalent to the standard } H^1 \text{norm },  \\
    \\
\left\{u\in\cW: \; \int_{M}\rho u\ d\mu_g =0\right\} & \text{otherwise}.
\end{cases}
\end{equation*} 

We consider the eigenvalue problem 
\begin{equation} \label{evprob} 
\int_M \rho u \overline{v} \ d\mu_g =\lambda\mathcal{E}_{g,\cW}[u,v],
\quad
u,v\in Z(\rho),
\end{equation}

Let us see that this problem is natural. If we assume for a moment that $\rho=1$, and $\cW=H^1(M)$, then $u\in Z(\rho)$ if and only if $u$ is orthogonal to the constant functions in the $L^2$ sense. On the other hand, a constant function is an eigenfunction of Laplace operator with the Neumann boundary condition corresponding to the eigenvalue zero.  Since the eigenfunctions of a self-adjoint operator are orthogonal, it follows that the eigenvalues of the Laplace operator on $Z(\rho)$ are the non-zero eigenvalues of Laplace operator with Neumann boundary condition.  In case $\cW=H_0^1(M)$, $Z(\rho)=H_0^1(M)$.  Moreover, the non-zero eigenvalues, $\lambda$, of \eqref{evprob} are in bijection with the eigenvalues $\Lambda$ of the weighted Laplace equation 
\begin{equation} \label{evprob2} \Delta_{g, \cW} u = \Lambda \rho u \end{equation} 
In this way, we refer to the eigenvalues of \eqref{evprob} as eigenvalues of a weighted Laplace equation. 

This type of equation arises in the study of hydrodynamics and elasticity, specifically in the linearisation of certain nonlinear problems; see  \cite{allegretto} and references therein.  Further motivation comes from quantum mechanics and the study of the behaviour of eigenvalues for Schr\"odinger operators with a large parameter; see \cite{gs}.  The equation
 $$(\Delta -qV(x))u=\lambda u, \quad q\to \infty$$
reduces to the study of the spectral problem 
$$\Delta u= \lambda V(x) u.$$
Above, $V$ is the electric potential, for which there are no reasons to assume that the sign is constant.  Consequently, it is quite interesting to study \eqref{evprob2} in the generality considered here, where we only assume the weight function $\rho \in L^\beta$ for some $\beta > \frac n 2$, but we do not assume that $\rho$ is of constant sign.  Investigation of this type of problem includes, but is not limited to  \cite{bcf}, \cite{bt}, \cite{hk}.  This equation is not only interesting and relevant to physics but also has applications in biology such as modelling population genetics; see  \cite{fleming}.

The main result of this paper is 
\begin{theorem}[\textbf{Weyl asymptotics for weighted Laplace equation with admissible boundary conditions}]
Let $M$ be a smooth compact manifold of dimension $\geq 2$ with smooth boundary, and let $g$ be a rough metric on $M$.  Then, the eigenvalues of \eqref{evprob} are discrete with finite dimensional eigenspaces with positive and negative eigenvalues, $\{ - \lambda_j ^- (\cW); \lambda_j ^+ (\cW) \}_{j=1} ^\infty$, such that 
$$- \lambda_1 ^- (\cW) \leq - \lambda_2 ^- (\cW) \leq \ldots < 0 < \ldots \leq \lambda_2 ^+ (\cW) \leq \lambda_1 ^+ (\cW).$$
Moreover, they satisfy the Weyl asymptotic formula
\begin{equation*}
\lim_{k\rightarrow\infty}\lambda_{k}^{\pm}(\cW)k^{\frac{2}{n}}=\left(\frac{\omega_n}{(2\pi)^{n}}\right)^{\frac{2}{n}}\left(\int_{M^{\pm}}|\rho(x)|^{\frac{n}{2}}\ d\mu_g\right)^{\frac{2}{n}}=\left(\frac{\omega_n}{(2\pi)^{n}}\right)^{\frac{2}{n}}\|\rho\|_{L^{\frac{n}{2}}(M^+, \ d\mu_g)}.
\end{equation*}
Above, $M^\pm := \{ x \in M : \pm \rho(x) > 0 \}$.  
\end{theorem}

As a corollary, we obtain classical Weyl asymptotics for the unweighted Laplace eigenvalue problem.

\begin{corollary}[Classical Weyl asymptotics] 
\label{cor:Main} 
Let $M$ be a smooth compact manifold of dimension $\geq 2$ with smooth boundary, and let $g$ be a rough metric on $M$.
Then, the  Laplacian $\Delta_{g,\cW}$ associated to an admissible boundary condition $\cW$ has discrete spectrum with finite dimensional eigenspaces, and
$$ \lim_{\lambda \to \infty}  \frac{N(\lambda, \Delta_{g,\cW})}{\lambda^{\frac{n}{2}}} = \frac{\omega_n}{(2\pi)^n} \mathrm{Vol}(M,g).$$
Above $N(\lambda, \Delta_{g,\cW})$ is the number of eigenvalues of $\Delta_{g,\cW}$ less than $\lambda$.
\end{corollary}


\subsection{Weyl's law in singular geometric settings} \label{s:sing-geom-set} 
Weyl's law has been previously demonstrated in many singular geometric settings.  Perhaps the most robust method for obtaining Weyl's law in these settings is the so-called heat kernel or semi-group method.  This method can be used to obtain Weyl's law on the manifolds with conical singularities studied by Cheeger \cite{cheeger}.  In that case, the Laplacian has terms with coefficients $r^{-2}$ with $r$ tending to $0$ at the conical singularity.  The heat kernel method can also be used to obtain Weyl's law for non-smooth spaces which arise as the limits of smooth, compact Riemannian manifolds.  Any sequence of smooth, compact Riemannian manifolds with Ricci curvature bounded below has a subsequence 
which converges in the pointed Gromov-Hausdorff sense to a limit space.  These limit spaces were studied in 1997--2000 by Cheeger and Colding \cites{cc1, cc2, cc3}.  In the non-collapsed case, they were able to define a Laplace operator on the limit space and obtain discreteness of its spectrum.  

In 2002,  Ding \cite{ding} used heat kernel techniques to obtain Weyl's law for the non-collapsed limit spaces studied by Cheeger and Colding.  Ding showed that the singular limit space has a well-defined heat kernel.  Relating this heat kernel to those for the smooth spaces, he could extract the Weyl law for the singular limit space.  More recently, Weyl's law has been studied in the context of metric spaces satisfying the Riemannian Curvature Dimension (RCD) condition \cite{RCDWeyl}.  Since it is impractical to provide an exhaustive list of references, we point the reader to the survey article \cite{Ivrii} by Ivrii and references therein, which provides an overview of RCD spaces and their development in a time-linear narrative.  The method used to prove Weyl's law in \cite{RCDWeyl} is also through the short time asymptotic behaviour of the trace of the heat kernel.   There is also a probabilistic approach via heat kernels; see \cite{AHT} by Ambrosio, Honda and Tewodrose where this method is described in the setting of RCD spaces. 

 Whereas the semi-group method can be used to demonstrate Weyl's law in both smooth as well as many singular settings, it is inaccessible in the rough metric setting.  To see this, recall that a key step in this method is to compare the heat kernel $H(t,x,y)$ to the function 
$$(4\pi t)^{-n/2} e^{-\frac{d(x,y)^2}{4t}}, $$
for small times, $t$, for points $x$ and $y$ which are sufficiently close.  Above $d(x,y)$ is the distance between the points $x$ and $y$ on the underlying space.  No such function can be defined without a well-defined notion of distance between points; hence this method is not available in the rough Riemannian manifold setting.  Our space is not obtained as a limit of smooth objects, so the approach of Ding \cite{ding} in the context of the limit spaces studied by Cheeger \& Colding \cites{cc1, cc2, cc3} is also not available.

\subsection{Strategy and structure of the paper} 
Since heat kernel methods are unavailable, as is any method which requires a well-defined notion of distance between points, we focus on abstract approaches rooted in functional analysis.  We are inspired by the work of the Soviet mathematicians, Birman and Solomjak \cite{BS}, who made a fundamental contribution to the study of the eigenvalue asymptotics for elliptic operators with  non-smooth  coefficients nearly  fifty  years ago.  In the present paper, which can be considered as a tribute to their achievements and a popularisation of their results, we demonstrate, for the particular case of  second-order operators in divergence form,  how their results can be carried over from the original setting for a domain in a Euclidean space to rather general compact manifolds with rough metrics.  Birman and Solomjak \cite{BS} obtained the principal terms of the asymptotic behaviour of the Dirichlet and Neumann problems for the equation $\mathcal{B}u=\lambda\mathcal{A}u$, where $\mathcal{A}$ is a self-adjoint elliptic operator, and $\mathcal{B}$ is a self-adjoint operator of lower order. In particular, they considered a generalised Dirichlet eigenvalue problem of the form
$$-\divv A \nabla u = \lambda B u$$ 
inside bounded Euclidean domains $\Omega\subset\R^n$. Above, $A(x)$ is a positive matrix for almost every $x\in \Omega$ such that $A^{-1}\in L^{\alpha}(\Omega)$, $A\in L^{\kappa}(\Omega)$, and $B\in L^{\beta}(\Omega)$ is a function which is in general not of constant sign. Under the conditions
\begin{equation*}
\alpha^{-1}+\beta^{-1}<2n^{-1},
\quad \alpha>n,
\quad \alpha^{-1}+\kappa^{-1}<2n^{-1},
\end{equation*}
they obtain the asymptotic behaviour of both positive and negative eigenvalues in terms of the $L^{\beta}$ norm of the function $B$. They mentioned, in \cite{BS}, that this result still hold for the case $\alpha^{-1}+\beta^{-1}=2n^{-1}$ and $n>2$, for the complete proof see \cite{Rozenblum1976}. Although their work is an invaluable technical tool, due to the different geometric setting and boundary conditions we consider, several additional results must be demonstrated to be able to apply \cite{BS}. 

This work is organised as follows.  In \S \ref{Sec:GeoPrel} we introduce rough metrics, describe their origin and connections with harmonic analysis, and give several examples. Then, in \S \ref{Sec:AnalPrel}, we show how a Laplace operator associated to a rough metric may be defined, and we introduce the admissible boundary conditions together with examples thereof.  Moreover, we demonstrate variational principles in the spirit of Courant, Rayleigh, and Poincar\'e for general eigenvalue problems like those considered here.  Our main results are proven in \S \ref{proofs}.  Concluding remarks are offered in \S \ref{Sec:CR}.  

\section*{Acknowledgements}
The first author was supported by the Knut and Alice Wallenberg foundation, KAW 2013.0322 postdoctoral program in Mathematics for researchers from outside Sweden, and by SPP2026 from the German Research Foundation (DFG).  The second author was partially supported by the Ministry of Education Science of the Republic of Kazakhstan under the grant AP05132071.  These authors also acknowledge the gracious support of the organisers of the event ``Harmonic Analysis of Elliptic and Parabolic Partial Differential Equations'' at CIRM Luminy as well as
the latter organisation.

\section{The rough Riemannian manifold setting}  \label{Sec:GeoPrel}
Throughout, we fix $M$ to denote a compact manifold of dimension equal to or exceeding $2$ with a smooth differentiable structure.
If the manifold has nonempty boundary $\partial M$, then we assume that it is smooth.
We let $T_x M$ and $T^\ast_x M$ be the tangent  and cotangent spaces at $x$, respectively, and $TM$ and $T^\ast M$ be the corresponding associated bundles.  The tensor bundles of covariant rank $q$ and contravariant rank $p$ are then denoted by $T^{(p,q)}M = (\oplus_{j=0}^p T^\ast M) \oplus (\oplus_{k=0}^q TM)$. 

In addition to a differentiable structure, such a space affords us with a notion of measurability independent of a {Riemannian} metric:  we say that a set $A$ is measurable if for every chart $(U,\psi)$ with $U \cap A \neq \emptyset$, we have that $\psi(A\cap U)$ is Lebesgue measurable in $\R^n$.  
We shall use $\cL$ to denote the Lebesgue measure in $\R^n$.  
Proposition 1 in \cite{BRough} shows that this notion of measurability is equivalent to asking for $A$ to be $\mu_h$-measurable, where $h$ is any smooth Riemannian metric on $M$, and $\mu_h$ is its induced volume measure.   
With this, we obtain a notion of a measurable section of a $(p,q)$ tensor.  
The set in which these objects live will be denoted by $\Gamma(T^{(p,q)}M)$. 

Similarly, we can define a measurable set $Z \subset M$ to be of zero measure if for every chart $(U,\psi)$,  when $U \intersect Z \neq \emptyset$, we have that $\psi(U \intersect Z)$ has zero Lebesgue measure.
This yields a notion of almost-everywhere in $M$ without alluding to a measure.
It is straightforward to verify that if $Z$ is of zero measure, then $\mu_h(Z) = 0$ for any smooth metric $h$.  Similarly, a property $P$ holds almost-everywhere precisely when $P$ holds $\mu_h$ almost-everywhere for any smooth metric $h$.

We can now present the precise notion of a \emph{rough metric}. 

\begin{definition}[Rough metric] \label{def:roughmetric} 
We say that a symmetric $(2,0)$ measurable tensorfield $g$ is a \emph{rough metric} if it satisfies the following \emph{local comparability condition}: for each $x \in M$, there exists a chart $(U_x,\psi_x)$ containing $x$ and a constant $C(U_x)  \geq 1$ such that 
$$ C(U_x)^{-1} \modulus{u}_{\psi_x^\ast \delta(y)} \leq \modulus{u}_{g(y)} \leq C(U_x) \modulus{u}_{\psi_x^\ast \delta(y)}$$
for almost-every $ y \in U_x$, for all $u \in T_y M$.
Above, $\psi_x^\ast \delta$ is the pullback to $U_x$ of the $\R^n$ scalar product inside $\psi(U_x)$. 
\end{definition} 

\begin{remark} \label{rmkrough}  As a consequence of the compactness of $M$, we note that the compatibility condition is equivalent to demanding that there exists a \em smooth Riemannian metric, $h$, \em on $M$ such that 
$$ C(U_x)^{-1} \modulus{u}_{h}  \leq \modulus{u}_{g} \leq C(U_x) \modulus{u}_{h}$$
 for almost-every $y \in U_x$,  where $U_x$, $u$, and $C(U_x)$ are as in Definition \ref{def:roughmetric}.
\end{remark}

Due to the regularity of the coefficients of a general rough metric $g$, it is unclear how to associate a canonical distance structure to $g$. However, the expression 
$$ \sqrt{ \det g(x)}\ d\psi_x^\ast \mathcal{L},$$
for almost-every $x \in U_x$ inside a compatible a chart $(U_x, \psi_x)$, can readily be checked to transform consistently under a change of coordinates. This yields a Radon measure that is independent of coordinates, which we denote by $\mu_g$.

\subsection{Rough metrics in harmonic analysis} \label{Sec:RMHM}  
These \emph{rough metrics} which are a focus in this paper were observed by Bandara in \cite{BRough} to be geometric invariances of the \emph{Kato square root problem} on manifolds without boundary.
 In a nutshell, this problem is to prove that $\dom(\sqrt{-\divv B \nabla}) = H^{1}(M)$ for \emph{bounded, measurable, complex, non-symmetric, elliptic} coefficient matrices, $x \mapsto B(x)$.  In the case of $M = \R^n$, this problem resisted resolution for over forty years.  It was finally settled by Auscher, Hofmann, Lacey, McIntosh, and Tchamitchian in \cite{AHLMcT}.  The first-order formulation of the problem by Axelsson\footnote{Andreas Axelsson is the former name of Andreas Ros\'en.}, Keith and McIntosh in \cite{AKM} allowed the problem to be considered in geometric settings.
Their approach was to obtain a solution to this problem by showing that an associated operator $\Pi_B$, which in part consists  of the
original operator in question, has an $H^\infty$ functional calculus.
In this paper, the authors obtained a solution of this problem on compact manifolds without boundary.  In the non-compact setting, Morris in \cite{Morris} first
obtained results in this direction for Euclidean submanifolds under second fundamental form bounds. 
Later, Bandara and McIntosh considered the intrinsic picture in \cite{BMc} and demonstrated that this problem can be solved under appropriate lower bounds on injectivity radius as well as Ricci curvature bounds.

A fundamental question of McIntosh was to understand the limitations of the methods used in this geometric version of the problem.  Exploiting the stability of the problem under $L^\infty$ perturbations, Bandara in \cite{BRough} showed that the problem could be solved far more widely than the previously used tools appeared to allow.  He showed that if the problem can be solved for some Riemannian metric $h$, then it also admits a solution for any rough metric $g$ which is $L^\infty$-close to $h$.  In this sense, rough metrics naturally emerged as geometric invariances of the Kato square root problem.  Indeed, the rough metrics as we have defined them here were introduced and investigated in \cite{BMc} and \cite{BRough} as geometric invariances of the Kato square root problem seen through the functional calculus of its  first-order characterisation.

It was shown in both \cite{AKM} by Axelsson (Ros\'en), Keith and McIntosh and later in \cite{BCont} by Bandara, that on smooth boundaryless compact manifolds, the Kato square root problem has a
positive solution. Counterexamples were first demonstrated by McIntosh in \cite{Mc71} and later adapted by Auscher in \cite{Auscher}. 
These counterexamples relied on having an operator whose spectrum grows exponentially.  Since the Kato square root problem can be solved in the boundaryless case for rough metrics, one may conjecture that the Laplacian associated to a rough metric ought to satisfy Weyl asymptotics.  Indeed, this was a key observation that prompted our investigation of the spectrum of the Laplacian on rough Riemannian manifolds with boundary.

\subsection{Geometric examples}
There are many natural examples of rough metrics, and here, we present here a small motivating collection.  
It is readily checked that every smooth or even continuous metric is rough. In particular, a metric of the form $g = \psi^\ast h$ is rough whenever $\psi: (M,h) \to (M,h)$ is a lipeomorphism, and $h$ is smooth.  
Recall that a \emph{lipeomorphism} is a homeomorphism that is also locally Lipschitz with locally Lipschitz inverse.
\begin{example}[Rough metrics arising from Lipschitz graphs]
Let $M$ be a smooth compact manifold with smooth boundary, and let $h$ be a smooth metric on $M$.  Let $N$ be some other smooth compact manifold with smooth boundary with metric $h'$.   Fix a Lipschitz function $f: M \to N$.  Note that 
$\Phi_f: M \to \graph(f) \subset M \times N$ given by $\Phi_f(x) = (x,f(x))$ is a lipeomorphism to its image.  Moreover, 
$$g_f(u,v) = (d\Phi_f(x)u, d\Phi_f(x)v)_{h \otimes h'},$$
defines a metric tensor on $M$.  Given the regularity of $f$, we have that $\Phi_f$ is a lipeomorphism.

To see that $g_f$ is a rough metric, fix $x \in M$, and let $\psi_x: U_x \to B(2,0)$ be a chart.
Letting $\sigma$ be a curve inside $U_x$, we obtain that 
$|\sigma'(t)|^2_{g_f}  = |\sigma'(t)|_{h}^2 + |df (\sigma'(t))|_{h'}^2.$
Now, 
$$
|df (\sigma'(t))|_{h'} \leq \sup_{{ t \in [0,1]}} |\mathrm{Lip}  f(\sigma(t)) | \modulus{\sigma'(t)}_{h} \leq {  \sup_{y \in U_x} |\mathrm{Lip} f(y)| \modulus{\sigma'(t)}_{h}}
$$
where 
$$ |\mathrm{Lip}f(y)| = \limsup_{z \to y} \frac{d_{h'}(f(y), f(z))}{d_{h}(y,z)}.$$
Since $f$ is a Lipeomorphism, this supremum is finite.  Thus $df$ is defined for almost every $t$, and it is a linear map between { $T_x M \to T_{f(x)} N$.}
Therefore, we have that for $\mu_g$ almost every $y \in U_x$ and $u \in T_y M$, setting $C(U_x) = \max\set{\sup_{U_x} |\mathrm{Lip} f|, 1}$ 
$$ C(U_x)^{-1} \modulus{u}_{h}  \leq \modulus{u}_{g_f} \leq C(U_x) \modulus{u}_{h}.$$
Consequently, $g_f$ is indeed a rough metric on $M$.

As a concrete example, let $M = B(1,0)$ be the unit ball with the Euclidean metric in $\R^n$, and $N = \R$ also with the standard Euclidean distance.  Then, $M \times N = B(1,0) \times \R \subset \R^{n+1}$, and the the metric tensor 
$h \otimes dt^2$ in this case is the usual Euclidean metric on $\R^{n+1}$.  If $f:B(1,0) \to \R$ a Lipschitz map, then $g_f$ defined as above is a rough metric on $B(1,0)$.  Although this example may seem contrite, we note that in the case $M=B(1,0) \subset \R^2$, given any set $E$ which has Lebesgue measure zero, there exists a Lipschitz function for which the singular set of this Lipschitz function, i.e. where it fails to be differentiable, contains the set $E$.  This set, $E$, can be a \em dense subset \em of $M=B(1,0)$.  We therefore see that even on the ball $B(1,0) \subset \R^2$, there exist non-trivial and highly singular rough metrics.

We remark that in general, we do not treat the case of Lipschitz boundary, but our methods apply to those manifolds with Lipschitz boundary which are lipeomorphic to a smooth manifold.  There are many Lipschitz manifolds that are not lipeomorphic or even homeomorphic to a smooth manifold.  In dimensions exceeding $4$, there are Lipschitz manifolds that do not admit a smooth structure.  This is seen by combining \cite{Kervaire} by Kervaire,  where he demonstrates the existence of topological manifolds without a smooth structure in dimensions exceeding $3$, and \cite{Sullivan} by Sullivan who shows that every topological manifold can be made into a Lipschitz manifold for dimensions exceeding $4$.  Although we do not treat these cases, our methods may be helpful for understanding these settings in the future.
\end{example}

\begin{example}[Manifolds with geometric cones]
Let $M$ be a smooth compact manifold with smooth boundary. Suppose that 
there are points ${x_1, \dots, x_k} \subset M$ and neighbourhoods $(\psi_i, U_i)$ of $x_i$
with $\psi_i(U_i) = B(1,0)$, the Euclidean ball of radius $1$.  Moreover, define maps $\Phi_i: U_i \to \R^{n+1}$ by 
$$\Phi_i(x) = (\psi_i(x), \cot(\alpha_i/2) (1 - |\psi_i(x)|_{\R^n})), \alpha_i \in (0, \pi].$$ 
A metric $g \in C^\infty(M \setminus \set{x_1, \dots, x_i})$ has geometric cones at 
$x_i$ if 
$$g(x) = \Phi_i^\ast(x)(\cdot,\cdot)_{\R^{n+1}}$$ 
inside $\psi_i^{-1}(B(1,0))$.

Such a cone point is, in fact, a conical singularity of angle $\alpha_i$.  That is, there exists a chart near $x_i$ such that the metric takes the form $g = dr^2 + \sin^2(\alpha_i) r^2 dy^2$. To see this, fix such a point $x_i$ with associated $\alpha_i \in (0,\pi]$.  
Let $(r, y) \in (0,1] \times S^{n-1}$ be polar coordinates in $B(1,0)$.  

Let $\gamma:I \to B(1,0)$ with $\gamma(0) \neq 0$, and let $\tilde{\gamma} = \psi_i^{-1}(\gamma)$.
Define 
$$\sigma(t) := (\Phi_i \circ \tilde \gamma)(t) = ( \gamma(t), \cot(\alpha_i/2)(1 - \gamma_r(t))),$$
where $\gamma_r (t) = | \gamma(t)|_{\R^n}$.  Therefore, 
$$\sigma'(t) = (\gamma'(t), -\cot(\alpha_i/2)\gamma_r'(t)).$$
In polar coordinates, $\gamma(t) = \gamma_r(t) \gamma_y(t)$, with $|\gamma_y(t)| = 1$.  We therefore compute that 
\begin{multline*}
\modulus{\tilde{\gamma}'(0)}_{g}^2 = (\sigma'(0), \sigma'(0))_{\R^{n+1}} 
	= \gamma_r^2(0) {\gamma_y'(0)}^2 + {\gamma_r'(0)}^2 + \cot^2(\alpha_i/2)^2 {\gamma_r'(0)}^2 \\
 	= \gamma_r^2(0) {\gamma_y'(0)}^2 + (1 + \cot^2(\alpha_i/2)) {\gamma_r'(0)}^2
 	= \gamma_r^2(0) {\gamma_y'(0)}^2 + \csc^2(\alpha_i/2)) {\gamma_r'(0)}^2. 
\end{multline*}

This shows that inside this chart, $g = \csc^2(\alpha_i/2) dr^2 + r^2 dy^2$.  A simple change to the coordinate
system $(r, y) \mapsto (\tilde{r}, y)$ given by $\tilde{r} = \csc(\alpha_i/2) r$, shows that in these coordinates, 
$g = d\tilde{r}^2 + \sin^2(\alpha_i/2) \tilde{r}^2 dy^2.$

The quintessential example in the situation with boundary is the standard cone of angle $\pi/2$, given by $M = \graph(f)$
where $f: B(1,0) \to \R$ is given by $f(x) = 1 - |x|$.  In the absence of boundary, the ``witch's hat sphere metric'' on the $n$-sphere $\mathbb{S}^n$ is a particular example of a manifold with a geometric cone which was considered by Bandara, Lakzian and Munn in  \cite{BLM}.\footnote{The term ``witch's hat'' arises from the Australian vernacular for a traffic cone.}  They studied a geometric flow tangential to the Ricci flow in a suitable sense that was defined by Gigli and Mantegazza in \cite{GM}.  An appealing feature of this flow is that it can be defined in many singular geometric settings such as metric spaces satisfying the RCD criterion. 
The regularity properties of this flow were not considered in \cite{GM}, which motivated the study of the flow in \cite{BLM} on the ``witch's hat sphere metric.''
This metric is the standard $n$-sphere metric away from a neighbourhood of the north pole, at the north pole, there is a geometric cone singularity of angle $\pi/2$. 
\begin{figure}
	\centering
	\includegraphics[width=40mm]{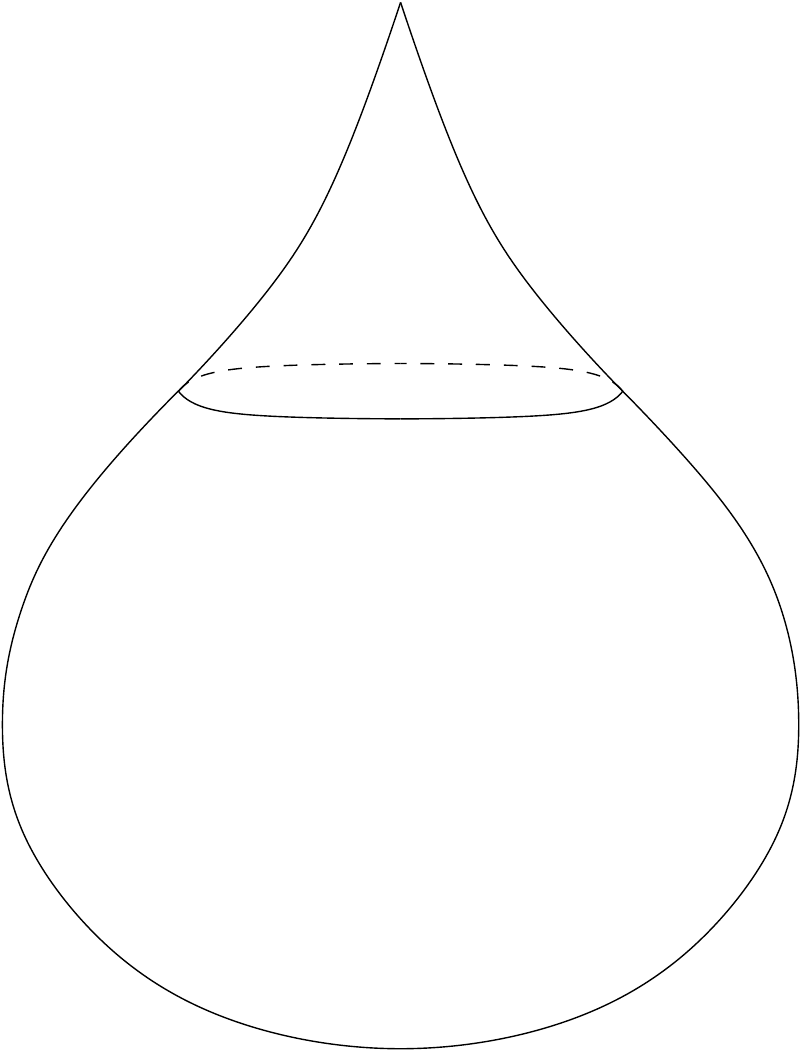}
	\caption{ {The witch's hat sphere metric.} }
	\label{fig:method}
\end{figure}
\end{example} 

\section{Analytic preliminaries}  \label{Sec:AnalPrel}
\subsection{Notation}
Throughout this paper, we assume the Einstein summation convention. That is, whenever  a raised index appears against a lowered index,  unless specified otherwise, we sum over that index.  
By $\#S$,  we denote the cardinality of a given set $S$.
In our analysis, we often write $a\lesssim b$ to mean that $a \leq Cb$, where $C>0$ is some constant. The dependencies of $C$ will either be explicitly specified or otherwise, clear from context. By $a\approx b$ we mean that $a\lesssim b$ and $b\lesssim a$.

\subsection{Dirichlet Forms and operators}\label{F_and_O}
Here we introduce some facts regarding closed symmetric densely defined  forms and the self-adjoint operators they generate. We let $\dom(\cdot)$ denote the domain of either an operator or a form.  

Let $\mathcal{H}$ be a separable Hilbert space with scalar product $(\cdot,\cdot)$, and $\mathcal{E}$ be a closed symmetric densely defined form in $\mathcal{H}$ such that
\begin{equation*}
\mathcal{E}[x,x]\gtrsim 0,
\qquad x\in \dom(\mathcal{E}).
\end{equation*}
Then, $\mathcal{E}$ generates a unique self-adjoint, non-negative operator $T$ in $\mathcal{H}$  with domain $\dom(T) \subset \dom(\mathcal{E})$ such that
\begin{equation*}
\mathcal{E}[x,y] = (Tx, y) 
\end{equation*}
for all $x \in \dom(T)$ and $y \in \dom(\mathcal{E})$. Moreover,
\begin{equation*}
\dom\left(T^{1/2}\right)=\dom(\mathcal{E})\quad\text{and}\quad \left(T^{1/2}x,T^{1/2}y\right)=\mathcal{E}[x,y]
\end{equation*}
for all $x$, $y\in \dom(\mathcal{E})$. 
If, additionally, the form $\mathcal{E}$ is strictly positive, i.e. $\mathcal{E}[x,x] \gtrsim \|x\|^2$, then $T$ is a strictly positive operator.
For a more detailed exposition of these results, see Theorem 2.1 and Theorem 2.32 in \S2 Chapter VI in \cite{Kato}.

Next, assume that $a$ is a completely continuous symmetric form in $\mathcal{H}$. Then it generates a unique completely continuous self-adjoint operator $A$ in $\mathcal{H}$ such that $a[u,v]=(Au,v)$ for $u$, $v\in \mathcal{H}$; see Section 2.2 in \cite{Stummel}.  We note that a completely continuous operator in a Hilbert space is compact.  
\begin{rem}\label{sense_of_eig_problem}
In this work we often investigate the eigenvalues of the problem
\begin{equation*}
a[u,v]=\lambda(u,v)
\quad
in
\quad
\mathcal{H},
\end{equation*}
by which we mean the eigenvalues of the unique operator $A$ such that $a[u,v] = (Au,v)$. Recall that every completely continuous self-adjoint operator is a compact self-adjoint operator and therefore has discrete spectrum accumulating at $0$; see Section 2.5 in \cite{Stummel}. 
\end{rem}

In the subsequent analysis, the following variational principles are indispensable tools.  The first variational principle is the so-called min-max characterisation of Courant; see \cite[Theorem 6.1]{Gorbachuk}. The second and third ones are called Poincar\'e's and Rayleigh's variational principles, respectively. We expect they are known but were unable to locate a proof in the generality required here, so we include the proofs.
\begin{theorem}\label{VP}
	Let $\mathcal{H}$ be a Hilbert space with scalar product $(\cdot,\cdot)$. Assume that $A$ is a completely continuous self-adjoint operator in $\mathcal{H}$, with the positive eigenvalues $\{\lambda_j^{+}\}_{j=1}^{\infty}$ and negative eigenvalues $\{-\lambda_j^{-}\}_{j=1}^{\infty}$ such that $\lambda_{k+1}^{\pm}\leq\lambda_{k}^{\pm}$. Let $\{u^{\pm}_j\}_{j=1}^{\infty}$ be the corresponding eigenfunctions. Then
	\begin{enumerate}[label=(\roman*)]
		\item\label{VPI} Courant's variational principle
		\begin{equation*}
		\lambda_{k}^{\pm}=\min_{L\subset \mathcal{H},\dim L^{\perp}=k-1}\max_{u\in L\setminus\{0\}}\pm\frac{(Au,u)}{(u,u)}.
		\end{equation*}
		\item\label{VPII} Poincar\'e's variational principles
		\begin{equation*}
		\lambda_{k}^{\pm}=\max_{V\subset \mathcal{H},\dim V=k}\min_{u\in V\setminus\{0\}}\pm\frac{(Au,u)}{(u,u)}.
		\end{equation*}
		\item\label{VPIII} Rayleigh's variational principles
		\begin{equation*}
		\lambda_k^{\pm}=\max\left\{\pm\frac{(Au,u)}{(u,u)}: \; u\in \{u^{\pm}_1,...,u^{\pm}_{k-1}\}^{\perp} \right\}.
		\end{equation*}
	\end{enumerate}
\end{theorem}
\begin{proof}
	As we mentioned above, we prove \ref{VPII} and \ref{VPIII}. We first derive \ref{VPII}. Since $\{u^{\pm}_j\}_{j=1}^{\infty}$ are the eigenfunctions corresponding to eigenvalues $\{\pm\lambda^{\pm}_j\}_{j=1}^{\infty}$, it follows
	\begin{equation}\label{max_min_th1ALT}
	\max_{V\subset \mathcal{H},\dim V=k}\min_{u\in V\setminus\{0\}}\pm\frac{(Au,u)}{(u,u)}\geq \min_{u\in \textrm{span}\{u_1^{\pm},...,u_k^{\pm}\}}\pm\frac{(Au,u)}{(u,u)}=\lambda_{k}^{\pm}.
	\end{equation}
	
	Let $L$ be the space for which the right side of \ref{VPI} is achieved. Then for any $k-$dimensional $V\subset \mH$, there exists $f\in L\cap V$.  Therefore
	\begin{equation*}
	\lambda_{k}^{\pm}=\max_{u\in L\setminus\{0\}}\pm\frac{(Au,u)}{(u,u)}\geq \pm\frac{(Af,f)}{(f,f)}\geq \min_{u\in V\setminus\{0\}}\pm\frac{(Au,u)}{(u,u)}.
	\end{equation*}
	Since this is true for all $k-$dimensional $V\subset \mathcal{H}$, we conclude
	\begin{equation*}
	\lambda_{k}^{\pm}\geq\max_{V\subset \mathcal{H},\dim V=k}\min_{u\in V\setminus\{0\}}\pm\frac{(Au,u)}{(u,u)}.
	\end{equation*}
	This together with \eqref{max_min_th1ALT} proves \ref{VPII}.
	
	Next, let us prove \ref{VPIII}. Let $\mathcal{H}^-$ and $\mathcal{H}^+$ be the spectral spaces corresponding to the negative and positive spectrum of $A$. Fix $u\in \mathcal{H}^{\pm}\cap \{u^{\pm}_1,...,u^{\pm}_{k-1}\}^{\perp}$, and assume $u=\sum_{j=1}^{\infty}a_ju_j^{\pm}$. Then $a_j=0$ for $j<k$ and
	\begin{equation*}
	\pm(Au,u)=\sum_{j=k}^{\infty}\lambda_j^{\pm}a_j^2\|u_j^{\pm}\|^2\leq \lambda_k^{\pm}\sum_{j=k}^{\infty}a_j^2\|u_j^{\pm}\|^2=\lambda_k^{\pm}\|u\|^2.
	\end{equation*}
	Note that for $u=u_k^{\pm}$ the inequality above becomes an equality. Therefore
	\begin{equation}\label{Rayleigh's_principle3ALT}
	\lambda_k^{\pm}=\max\left\{\pm\frac{(Au,u)}{(u,u)}: \; u\in \mathcal{H}^{\pm}\cap\{u^{\pm}_1,...,u^{\pm}_{k-1}\}^{\perp} \right\}.
	\end{equation}
	Let $f\in\{u^{\pm}_1,...,u^{\pm}_{k-1}\}^{\perp}$ and $f^{\pm}$ be projections of $f$ into $\mathcal{H}^{\pm}$. Then we write
	\begin{equation*}
	(Af,f)=(Af^-,f^-)+2(Af^+,f^-)+(Af^+,f^+).
	\end{equation*}
	Hence, since $\pm(Af^{\pm},f^{\pm})\geq 0$ and $(Af^+,f^-)=0$, we derive $\pm(Af,f)\leq \pm(Af^{\pm},f^{\pm})$. Therefore the right hand side of \ref{VPIII} does not exceed the right hand side of \eqref{Rayleigh's_principle3ALT}. On the other hand, the right side of \eqref{Rayleigh's_principle3ALT} by its very definition does not exceed the right side of \ref{VPIII}.
\end{proof}

\subsection{Laplacian associated to admissible boundary conditions}
\label{S:Lap}
As aforementioned, a rough metric has a canonically associated Radon measure, $\mu_g$, and so we may define $L^k(T^{(p,q)} M,\ d\mu_g)$ spaces in the usual way.  
The Sobolev spaces $H^k(M)$ and $H^k_0(M)$ on a compact Riemannian manifold with boundary are independent of the metric.
The central issue for us is to ensure that $H^1(M)$ and $H^1_0(M)$ agree with the domains of the defining operators for us in our Dirichlet forms.
For this, we need the fact that 
$\nabla_2 = \nabla = d:  C^\infty \cap L^2(M,\ d\mu_g) \to C^\infty \cap L^2(T^\ast M, \ d\mu_g)$ is a closable operator.
This uses the fact that the exterior derivative $d$ depends only on the differential structure of $M$ as well as the properties of the rough metric.
See \cite{BRough} for details. 
Armed with this fact, we assert that $H^1(M) = \dom( \overline{\nabla_2})$.
Moreover, we also consider $\nabla_c = \nabla: C^\infty_c(M) \to C^\infty_c(M)$. 
Since $\dom(\nabla_c) \subset \dom(\nabla_2),$ we obtain $H^1_0(M) = \dom (\overline{\nabla_c}) = \overline{C^\infty_c}^{\norm{\cdot}_{H^1}}$.  
In the situation that $\partial M = \emptyset$, we obtain that $H^1_0(M) = H^1(M)$. 

Recall that in the case of a smooth metric, the Laplacian obtained by the Dirichlet forms $\mathcal{E}_{N}[u,v] = (\nabla u, \nabla v)_{L^2(T^\ast M, \ d\mu_g)}$  with $\dom(\mathcal{E}_N) = H^1(M)$ and  $\mathcal{E}_D[u,v] = \mathcal{E}_N[u,v]$ but with 
domain $\dom(\mathcal{E}_D) = H^1_0(M)$ respectively yield the Neumann and Dirichlet Laplacians (c.f. \S4 and \S7 in \cite{IS}).
To define Neumann and Dirichlet boundary conditions in the classical setting, we require that the metric
induces a surface measure and has an accompanying Stokes's Theorem. However, in our context of a rough metric, since the coefficients are assumed to be only measurable, 
it is unclear how to extract a surface measure despite the fact that our boundary is smooth.
However, the Dirichlet form perspective for both these problems persist in our setting and hence, we will retain the nomenclature and call
the Laplacians obtained by these energies respectively Neumann and Dirichlet.

In fact, in this paper, we will consider more general boundary conditions where the Dirichlet and Neumann Laplacians are
the extreme ends. Namely, we introduce 

\begin{definition}[Admissible boundary condition] 
Let $\cW \subset H^1(M)$ be a closed subspace of $H^1(M)$ such that $H^1_0 \subset \cW$. 
Then we call $\cW$ an admissible boundary condition. 
\end{definition}

Define the Dirichlet form, $\mathcal{E}_{g,\cW}: \cW \times \cW \to \C$ associated to $\cW$ by
$$
\mathcal{E}_{g,\cW}[u,v] = (\nabla u, \nabla v)_{L^2(T^\ast M, \ d\mu_g)}.$$
From the representation theorems, namely, Theorem 2.1 in \S2  and Theorem 2.23 in \S6 in Chapter IV of \cite{Kato} by Kato, 
we obtain the Laplacian, $\Delta_{g,\cW}$.  It is a non-negative  self-adjoint operator
with domain $\dom(\Delta_{g,\cW}) \subset \cW$ and with $\dom(\sqrt{\Delta_{g,\cW}}) = \cW$.
Defining $\nabla_{\cW}$ as the operator $\nabla $ with domain $\cW$, we see that it is a closed operator and hence obtain a densely-defined and closed adjoint $\nabla_{\cW,g}^\ast$ by Theorem 5.29 in \S5 of Chapter 3 in \cite{Kato}.  A routine operator theory argument yields
$$\Delta_{g,\cW} = \nabla^{\ast}_{g,\cW} \nabla_{\cW}.$$

\begin{prop} \label{discretespec} 
The spectrum of $\Delta_{g, \cW}$ is discrete with no finite accumulation points and with each eigenspace being of finite dimension.
\end{prop} 
\begin{proof} 
Since we have assumed that the boundary of $M$ is smooth, we have by 
the Rellich-Kondrachov theorem (c.f. Theorem 2.34 in \cite{Aubin}) that $H^1(M) \embed L^2(M,\ d\mu_g)$ compactly.
Therefore, we can factor the resolvent $(i  + \Delta_{g,\cW}): L^2(M,\ d\mu_g) \to L^2(M,\ d\mu_g)$ as 
$$(i + \Delta_{g,\cW}): L^2(M,\ d\mu_g) \to \dom(\Delta_{g,\cW}) \to \dom(\sqrt{\Delta_{g,\cW}}) = \cW \to H^1(M) \embed L^2(M,\ d\mu_g).$$
Hence, we obtain that this is a compact map. 

A sufficient condition for a self-adjoint operator $T$ to have discrete spectrum is for $(\zeta - T)^{-1}$ to be a compact operator, for some $\zeta$ in the resolvent set. In this situation, one also has that the operator $T$ has no finite accumulation points, and that each eigenspace is finite dimensional; this follows from Theorem 6.29 in \S6 in Chapter III in \cite{Kato}.  By Theorem 5.29 in \S 5 of Chapter 3 in \cite{Kato}, $\Delta_{g,\cW}$ is self-adjoint, and $i$ is an element of the resolvent set.  
\end{proof}

\subsection{Examples of boundary conditions}
We shall see that Dirichlet, Neumann, and mixed boundary conditions are all admissible.  
\begin{example}[Dirichlet and Neumann conditions]
As mentioned in \S\ref{S:Lap}, for any smooth compact manifold $M$ with smooth boundary and a smooth metric $g$, 
$\cW = H^1_0(M)$ corresponds to Dirichlet boundary condition, and $\cW = H^1(M)$ is the Neumann counterpart.
This is easily verified by Stokes's theorem, and by using the existence of a unit outer normal to the boundary.
See \S4 and \S7 in \cite{IS} for the calculation in the case of Euclidean domains.
\end{example} 

\begin{example}[Mixed boundary conditions]
Let $(M,g)$ be a smooth compact manifold and assume also that $g$ is smooth. 
Fix $\Sigma \subset \partial M$
a closed subset of the boundary $\partial M$ with nonempty interior. Then, define 
$$\cW = \set{u \in H^1(\Omega): \spt u\rest{\partial M} \subset \Sigma\ \text{a.e. in}\ \partial M}.$$

This is a closed subspace of $H^1(M,g)$. To see this, let $u_n \in \cW$ converge
to $u \in H^1(M)$. Then, letting $\Sigma^c = \partial M \setminus \Sigma$, 
\begin{multline*} 
\|u\rest{\partial M}\|_{L^2(\Sigma^c)} 
	\leq \|u\rest{\partial M} - (u_n)\rest{\partial M}\|_{L^2(\Sigma^c)} +  \| (u_n)\rest{\partial M}\|_{L^2(\Sigma^c)}\\
	\leq \|u\rest{\partial M} - (u_n)\rest{\partial M}\|_{L^2(\partial M)} 
	\leq \|u - u_n\|_{H^1(M)}
\end{multline*}
where the penultimate inequality follows from the fact that $(u_n)\rest{\partial M} = 0$ almost-everywhere in $\Sigma^c$, 
whereas the ultimate from the boundedness of the trace map 
$$u \mapsto u\rest{\partial M}: H^1(M) \to H^{\frac{1}{2}}(\partial M) \embed L^2(\partial M).$$ 
By letting $n \to \infty$, this shows that $u\rest{\partial M} = 0$ almost-everywhere in $\Sigma^c$  and hence $u \in \cW$.   The Laplacian $\Delta_{g,\cW}$ then has mixed-boundary conditions, with Dirichlet boundary conditions on $\partial M \setminus \Sigma$ and Neumann on $\Sigma$. 

When $M = \Omega \subset \R^n$, a bounded domain, with $g$ as the standard Euclidean metric were
considered by Axelsson (Ros\'en), Keith and McIntosh in \cite{AKM2}  to study the Kato square root problem under mixed boundary conditions.
As mentioned in \S\ref{Sec:RMHM}, the Kato square root problem was phrased from a first-order framework and obtained via showing that an associated operator $\Pi_B$ has a $H^\infty$ functional calculus. 
In the presence of boundary, the operator $\Pi_B$ does not have a canonical domain.
The domains considered are built from closed subspaces $\cV \subset H^1(\Omega)$  satisfying $H^1_0(\Omega) \subset \cV$.
The two extremes, $\cV = H^1_0(\Omega)$ and $\cV = H^1(\Omega)$ correspond to the Dirichlet and Neumann conditions for the relevant part of the operator $\Pi_B ^2$ respectively, where in the functional calculus, this operator is accessed by simply taking a relevant function $f$ and considering a new function $z \mapsto f(z^2)$.
The conditions $\cW$ which we have defined above, in this context, are precisely the ``mixed-boundary conditions'' of \cite{AKM2}. 
\end{example} 


\section{Proof of the main results} \label{proofs}
\subsection{The statement of the problem}\label{stetment_of_the_problem}
For the convenience of the reader, we recall the key notions from the introduction.  Let $M$ be a smooth compact manifold of dimension $n\geq 2$ with smooth boundary, and let $g$ be a rough metric on $M$.  Let $\beta>\frac{n}{2}$, and $\rho\in L^{\beta}(M, \ d\mu_g)$ be a real valued function such that
\begin{equation*}
\int_{M}\rho \ d\mu_g\neq 0.
\end{equation*}
We consider the Dirichlet form
\begin{equation*}
\mathcal{E}_{g,\cW}[u,v]=\left(\nabla u,\nabla v\right)_{L^2(M,\mathbb{C}^n,\ d\mu_g)}.
\end{equation*}

Let $\cW$ be an admissible boundary condition.  Associated to the Dirichlet form is a subspace of $\cW$, 
\begin{equation*}
Z(\rho)=
\begin{cases}
\cW & \text{if }\mathcal{E}_{g,\cW} \text{ generates the norm in } \cW,\\
    & \text{which is equivalent to the standard } H^1(M) \text{ norm },  \\
    \\
\left\{u\in\cW: \; \int_{M}\rho u\ d\mu_g =0\right\} & \text{otherwise}.
\end{cases}
\end{equation*}
Note, as we mentioned in the introduction, for $\rho=1$, $Z(\rho)$ is the intersection of $\cW$ and the closure of the operator $\Delta_{g, \cW}$. 

\begin{prop} \label{p:zrho}
	The subspace $Z(\rho)\subset\cW$ is closed in $H^1(M)$ norm. Moreover, $\dim Z(\rho)^{\perp}=\tau \leq 1$, where orthogonality is in the $\cW$ sense.
\end{prop}

\begin{proof}\label{closed_subspace}

		Let us choose $\frac{2\beta n}{n\beta-n+2\beta}<q<2\leq n$, then $\frac{nq}{n-q}<\frac{2\beta}{\beta-1}$. Therefore The Sobolev Embedding Theorem, see \cite[Theorem 10.1]{Hebey}, gives the  continuous embeddings
	\begin{equation}\label{continuous_emb}
	H^1(M)=W^{1,2}(M)\subset W^{1,q}(M)\subset
	L^{\frac{2\beta}{\beta-1}}(M)\subset
	L^{\frac{\beta}{\beta-1}}(M),
	\end{equation}
	where the second embedding is compact. Therefore the H\"{o}lder inequality gives
	\begin{equation*}\label{ALTEclosed_subspace_1}
	\left|\int_{M}\rho u \ d\mu_g \right|\lesssim \|\rho\|_{L^{\beta}}\|u\|_{L^{\frac{\beta}{\beta-1}}}\lesssim \|\rho\|_{L^{\beta}}\|u\|_{H^1}
	\end{equation*}
for $u\in H^1(M)$. The implicit constants depend only on the volume of $M$ with respect to $\ d\mu_g$.
	In the case $Z(\rho) = \cW$, obviously $Z(\rho)$ is closed in $\cW$.  So, let us assume we are in the second case. Let $\{f_j\}_{j=1}^{\infty}\subset Z(\rho)$ and $f\in \cW$ such that $f_n\rightarrow f$ in the $H^1$ norm. Since $\int_{M}\rho f_j\ d\mu_g=0$, we obtain
	\begin{equation*}
	\left|\int_{M}\rho f \ d\mu_g\right|=\left|\int_{M}(\rho f-\rho f_j) \ d\mu_g \right|\lesssim \|\rho\|_{L^{\beta}}\|f-f_j\|_{H^1}.
	\end{equation*}
	By letting $j\rightarrow\infty$, we conclude that $\int_{M}\rho f \ d\mu_g=0$, and hence $f\in Z(\rho)$, so $Z(\rho)$ is a closed subspace of $\cW$ with respect to the $H^1(M)$ norm.
	
	To prove the statement regarding the dimension of $Z(\rho)$, we assume for the sake of contradiction that $\dim Z(\rho)^{\perp}=\tau >1$.  Then there exists linearly independent $v_1$, $v_2\in Z(\rho)^{\perp}$. Since $Z(\rho)^\perp$ is a subspace, any linear combination of $v_1$ and $v_2$ should also be in $Z(\rho)^\perp$.  Let $c_j=\int_{M}\rho v_j \ d\mu_g$ for $j=1,2$. Note that $c_j\neq0$, so  there exists $a\in \R$ such that $ac_1-c_2=0$, which is equivalent to
	\begin{equation*}
	a\int_{M}\rho v_1\ d\mu_g-\int_{M}\rho v_2\ d\mu_g=0 = \int_M \rho(a v_1 - v_2) \ d\mu_g.
	\end{equation*}
	Therefore $av_1-v_2 \in Z(\rho)$.  However, since $a v_1 - v_2$ is a linear combination of $v_1, v_2 \in Z(\rho)^{\perp}$, we have that $av_1 - v_2 \in Z(\rho) \cap Z(\rho)^\perp$.   This means that $av_1 - v_2 = 0$ which contradicts the linear independence of $v_1$ and $v_2$.  
\end{proof}

\begin{remark}[Notational simplifications]  We may, for the sake of simplicity use $L^k$ to denote $L^k (M, \ d\mu_g)$.  We shall do this when we are only working with respect to the measure, $\ d\mu_g$.  In case we are working with different measures, we shall use the more cumbersome notation to indicate the measure of integration.
\end{remark}

\begin{lemma}[Poincar\'e inequality]\label{Poinare_inequality}
	There exists a constant $C>0$ such that
	\begin{equation}\label{Poinare_inequality1}
	\left\|u-\frac{1}{\int_{M}\rho \ d\mu_g}\int_{M}\rho u \ d\mu_g \right\|_{L^2}\leq C\|\nabla u\|_{L^2}
	\end{equation}
	holds for any $u\in H^1(M)$.
\end{lemma}

\begin{proof}
	Without lost of generality, assume that $\int_{M}\rho \ d\mu_g=1$. Assume that \eqref{Poinare_inequality1} false for every $C>0$. Then there exists a sequence of functions $\{u_j\}_{j=1}^{\infty}$ such that the left side of \eqref{Poinare_inequality1} equals $1$ for all $j$ while the right hand side tends to zero as $j\rightarrow\infty$. Let 
	$$h_j=u_j-\int_{M}\rho u_j \ d\mu_g.$$ 
	Since $\|h_j\|_{L^2}=1$, and $\|\nabla h_j\|_{L^2}=\|\nabla u_j\|_{L^2}$, the sequence $\{h_j\}_{j=1}^{\infty}$ is bounded in $H^1(M)$. Therefore, since every bounded sequence in a Hilbert space has a weakly convergent subsequence, we may assume that there exists $h\in H^1(M)$ such that $h_j\rightarrow h$ in weakly in $H^1(M)$. Since $\|\nabla h_j\|_{L^2}\rightarrow 0$ as $j\rightarrow\infty$, we obtain that $\nabla h=0$ in the distributional sense. Since $M$ is connected, it follows that $h$ is constant function. Since $\int_{M}\rho h_j\ d\mu_g=0$ and the second embedding in \eqref{continuous_emb} is compact, we derive
\begin{equation*}
\left|\int_{M}\rho h \ d\mu_g \right|=\left|\int_{M}\rho h \ d\mu_g - \int_{M}\rho h_j \ d\mu_g\right|\leq\|\rho\|_{L^{\beta}}\|h-h_j\|_{L^{\frac{\beta}{\beta-1}}}\rightarrow 0
\end{equation*}

as $j\rightarrow\infty$. Since $\int_M \rho \ d\mu_g \neq 0$, and $h$ is constant, $h=0$. This contradicts $||h||_{L^2} = 1$.  
\end{proof}

\begin{cor}\label{equivalence_of_norms}
	The form $\mathcal{E}_{g,\cW}[\cdot,\cdot]$ generates the norm in $Z(\rho)$, which is equivalent to the standard $H^1(M)$ norm.
\end{cor}
\begin{proof}
The proof follows from Lemma \ref{Poinare_inequality} and Proposition \ref{p:zrho}.
\end{proof}

By Proposition \ref{equivalence_of_norms}, $Z(\rho)$, equipped with the norm $\mathcal{E}_{g,\cW}[\cdot,\cdot]$, is a Hilbert space, which we denote by $\left(Z(\rho), \mathcal{E}_{g,\cW}\right)$.

Let us consider the form
\begin{equation*}
\rho[u,v]=\int_M \rho u \overline{v}\ d \mu_g
\quad \text{in}
\quad \left(Z(\rho), \mathcal{E}_{g,\cW}\right).
\end{equation*}
In order to see that $\rho[\cdot,\cdot]$ is well defined, recall 
\begin{equation}\label{compact_emb}
	\left(Z(\rho), \mathcal{E}_{g,\cW}\right)\subset H^1(M)\subset L^{\frac{2\beta}{\beta-1}}(M),
\end{equation}
where the first embedding is continuous by Corollary \ref{equivalence_of_norms} and the second embedding is compact by \eqref{continuous_emb}. Hence, by \eqref{continuous_emb}, the H\"{o}lder inequality implies
	\begin{equation}\label{H3ALT}
	|\rho[u,v]|\leq \|\rho\|_{L^{\beta}}\|u\|_{L^{\frac{2\beta}{\beta-1}}}\|v\|_{L^{\frac{2\beta}{\beta-1}}},
	\end{equation}
for $u$, $v\in Z(\rho)$, so that $\rho[\cdot,\cdot]$ is well defined. Moreover, the next proposition holds.

\begin{prop}\label{rho_is_comp_con}
	The form $\rho[\cdot,\cdot]$ is a completely continuous form in the Hilbert space $\left(Z(\rho), \mathcal{E}_{g,\cW}\right)$.
\end{prop}
\begin{proof}	
Let $u_j\rightarrow u$ and $v_j\rightarrow v$ weakly in $\left(Z(\rho), \mathcal{E}_{g,\cW}\right)$. Then, by \eqref{H3ALT}, we estimate
	\begin{align*}
       \left| \rho [u,v]-\rho [u_j,v_j] \right|& = \left| \rho [u,v-v_j]-\rho [u_j-u,v_j] \right|\\
       &\lesssim  \|\rho\|_{L^{\beta}}\|u\|_{L^{\frac{2\beta}{\beta-1}}}\|v-v_j\|_{L^{\frac{2\beta}{\beta-1}}} + \|\rho\|_{L^{\beta}}\|u_j-u\|_{L^{\frac{2\beta}{\beta-1}}}\|v_j\|_{L^{\frac{2\beta}{\beta-1}}}.
    \end{align*}
     Since \eqref{compact_emb}, $u_j\rightarrow u$ and $v_j\rightarrow v$ strongly in $L^{\frac{2\beta}{\beta-1}}(M)$. Hence the last estimate implies that $\rho [\cdot,\cdot]$ is a completely continuous form in $\left(Z(\rho), \mathcal{E}_{g,\cW}\right)$.
\end{proof}

Proposition \ref{rho_is_comp_con} allows us to consider eigenvalues of the problem
\begin{equation} \label{evprob00} 
\rho [u,v]=\lambda\mathcal{E}_{g,\cW}[u,v],
\quad
u,v\in Z(p),
\end{equation}
in the sense of Remark \ref{sense_of_eig_problem}. We henceforth denote the non-zero eigenvalues of eigenvalue problem \eqref{evprob00} by 
$$\{-\lambda_{j}^{-}(\cW);\lambda_{j}^{+}(\cW)\}_{j=1}^{\infty},$$ 
such that
\begin{equation*}
-\lambda_1^-(\cW)\leq-\lambda_2^-(\cW)\leq...<0<...\leq\lambda_2^+(\cW)\leq\lambda_1^+(\cW).
\end{equation*}
Our main task is to investigate the asymptotic behaviour of these eigenvalues. We begin by recalling results obtained by Birman and Solomjak \cite{BS} for domains in $\R^n$.

\subsection{Eigenvalue asymptotics for the weighted Laplace equation on Euclidean domains}\label{Birman_Salomyak}
Here we state the simplified version of Theorems 3.2 and 3.5 in \cite{BS}, and modify them for our propose.

Let $\Omega\subset\R^n$ be a domain with Lipschitz boundary. Assume that, for almost every $x\in \Omega$, $B(x)$ is a positive number such that $B^{-1}$, $B\in L^{\infty}(\Omega)$, and $A(x)$ is a positive $n\times n$ matrix such that $A^{-1}$, $A\in L^{\infty}(\Omega)$. Let $P\in L^{\beta}(\Omega)$, where $\beta>n/2$. 

In this subsection, we assume that $\cW$ is one of the spaces $H_0^1(\Omega)$ or $H^1(\Omega)$. Consider the following forms, for $t>0$,
\begin{equation*}
\begin{gathered}
a_t[u,v]:=(A\nabla u,\nabla v)_{L^2(\Omega, d\cL)}+t(u,v)_{L_2(\Omega, d\cL)},
\quad \dom(a_t)=\cW,\\
a_B[u,v]:=(A\nabla u,\nabla v)_{L^2(\Omega, d\cL)}+(Bu,v)_{L_2(\Omega, d\cL)},
\quad \dom(a_B)=\cW.
\end{gathered}
\end{equation*}
Let $(\cW, a_t)$ and $(\cW, a_B)$ be the spaces of functions $u\in \cW$ equipped with the norms $a_t[\cdot,\cdot]$  and $a_B[\cdot,\cdot]$ respectively. Since $t>0$, and $A(x)$ is positive for almost every $x\in \Omega$, the norms $a_t[\cdot,\cdot]$  and $a_B[\cdot,\cdot]$ are equivalent to the standard norm in $H^1$. Therefore $(\cW, a_t)$, $(\cW, a_B)$ are the Hilbert spaces, and they are equal to $\cW$ as a set.

Consider the form 
\begin{equation*}
p[u,v]=\int_{\Omega}P u \overline{v} \ d\cL
\quad \text{in}
\quad \cW.
\end{equation*}
By the same arguments we do in Proposition \ref{rho_is_comp_con}, this form, $p[\cdot,\cdot]$, is a completely continuous symmetric form in both Hilbert spaces $(\cW, a_t)$ and $(\cW, a_B)$. Therefore, in the sense of Remark \ref{sense_of_eig_problem}, the eigenvalue problems
\begin{equation}\label{t_prob_Omega}
p[u,v]=\lambda a_t[u,v], \quad in \quad \cW,
\end{equation}
\begin{equation}\label{B_prob_Omega}
p[u,v]=\lambda a_B[u,v], \quad in \quad \cW,
\end{equation}
have the discrete spectrum, eigenvalues with finite multiplicity, and accumulating at $0$. Let us denote the non-zero eigenvalues of \eqref{t_prob_Omega} and \eqref{B_prob_Omega} by $\{-\lambda_j^-(a_t), \lambda_j^+(a_t)\}_{j=1}^{\infty}$ and $\{-\lambda_j^-(a_B), \lambda_j^+(a_B)\}_{j=1}^{\infty}$, respectively, ordered such that
\begin{equation*}
\begin{gathered}
-\lambda_1^-(a_t)\leq -\lambda_2^-(a_t)\leq...<0<...\leq\lambda_2^+(a_t)\leq \lambda_1^+(a_t),\\
-\lambda_1^-(a_B)\leq -\lambda_2^-(a_B)\leq...<0<...\leq\lambda_2^+(a_B)\leq \lambda_1^+(a_B).
\end{gathered}
\end{equation*}
Let $N^{\pm}(\lambda,a_t)$ and $N^{\pm}(\lambda,a_B)$ be the distribution functions of eigenvalues of problems \eqref{t_prob_Omega} and \eqref{B_prob_Omega} respectively, 
\begin{equation*}
N^{\pm}(\lambda,a_t)=\# \left\{\lambda_j^{\pm}(a_t)>\lambda\right\},
\quad
N^{\pm}(\lambda,a_B)=\# \left\{\lambda_j^{\pm}(a_B)>\lambda\right\}.
\end{equation*}
The eigenvalues above are counted according to multiplicity.  

The following theorem, in a more general form, was proved in \cite{BS}. We state here the simple version which shall be an essential ingredient in the proof of our main result.   

\begin{thm}[Theorem 3.2 and 3.5 in \cite{BS}]\label{as_for_eig_prob_in_euq}
	We have the asymptotic formulas
	\begin{equation*}
	\lim_{\lambda\rightarrow 0}\lambda \left(N^{\pm}(\lambda,a_t)\right)^{\frac{2}{n}}
	=\left(\frac{\omega_n}{(2\pi)^n}\int_{\Omega^{\pm}}\frac{|P(x)|^{\frac{n}{2}}}{\sqrt{\mathrm{det}A(x)}}dx\right)^{\frac{2}{n}},
	\end{equation*}
	where $\Omega^{\pm}:=\{x\in \Omega: \; \pm P(x)\geq0\}$, and $\omega_n$ is the volume of the unit ball in $\R^n$.
\end{thm}

We shall use the preceding result to obtain an asymptotic formula for $N^{\pm}(\lambda,a_B)$.
\begin{theorem}\label{as_for_eig_prob_in_euq_B}
	We have the asymptotic formulas
	\begin{equation*}
	\lim_{\lambda\rightarrow 0}\lambda \left(N^{\pm}(\lambda,a_B)\right)^{\frac{2}{n}}
	=\left(\frac{\omega_n}{(2\pi)^n}\int_{\Omega^{\pm}}\frac{|P(x)|^{\frac{n}{2}}}{\sqrt{\mathrm{det}A(x)}}dx\right)^{\frac{2}{n}},
	\end{equation*}
	where $\Omega^{\pm}:=\{x\in \Omega: \; \pm P(x)\geq0\}$, and $\omega_n$ is the volume of the unit ball in $\R^n$.
\end{theorem}

\begin{proof}
		Since $\Omega\subset\R^n$ is a bounded domain with piecewise smooth boundary, the Sobolev Embedding Theorem gives that $\mathrm{id}:\cW\hookrightarrow L^2(\Omega, d\cL)$ is compact. Therefore, the multiplication operator 
		$$
		(B-t):\cW\to L^2(\Omega, d\cL)
		$$
		can be factored as an operator $\cW \stackrel{\mathrm{id}}{\embed} L^2(\Omega, d\cL) \to L^2(\Omega, d\cL)$ which shows that it is compact. In particular, this guarantees that it is a completely continuous map.
		Letting $\mathcal{E}[u,v] = a_B[u,v] - a_t[u,v]$, and taking $u_j \to u$ and $v_j \to v$ weakly,
		\begin{multline*} 
		\mathcal{E}[u,v] - \mathcal{E}[u_j, v_j] = ((B-t)u,v)_{L_2(\Omega, d\cL)}-((B - t) u_j, v_j)_{L_2(\Omega, d\cL)} \\ =( (B-t)u, v - v_j)_{L_2(\Omega, d\cL)} - ((B-t)(u_j - u), v_j)_{L_2(\Omega, d\cL)}.
		\end{multline*}
		By an application of the Banach-Steinhaus theorem, we can deduce that $\norm{v_j} \lesssim 1$ since it is weakly convergent. 
		Therefore, 
		$$|((B-t)(u_j - u), v_j)_{L_2(\Omega, d\cL)}| \lesssim \norm{(B -t)(u_j - u)},$$
		and this tends to zero by the complete continuity of $(B -t)$.  
		The remaining term tends to zero by the fact that $v_j \to v$ weakly, and therefore, $\mathcal{E}$ is a completely continuous Dirichlet form on $\cW$. Finally, \cite[Lemma 1.3]{BS} implies that
		\begin{equation*}
		\lim_{\lambda\rightarrow 0}\lambda \left(N^{\pm}(\lambda,a_B)\right)^{\frac{2}{n}}
		=\lim_{\lambda\rightarrow 0}\lambda \left(N^{\pm}(\lambda,a_t)\right)^{\frac{2}{n}}
		\end{equation*}
		and hence Theorem \ref{as_for_eig_prob_in_euq} implies the statement.
\end{proof}

\subsection{An auxiliary problem}
We shall demonstrate a Weyl asymptotic formula for a Dirichlet form in the spirit of $a_t$.  This will then be used to obtain our main result.   
Let $\cW$ be an admissible boundary condition. Consider the following Dirichlet form, for $t>0$,
\begin{equation*}
\mathcal{E}_{g,\cW,t}[u,v]=\mathcal{E}_{g,\cW}[u,v]+t(u,v)_{L^2(M, \ d\mu_g)},
\quad in \quad \cW.
\end{equation*}
We are interested in the eigenvalues, $\nu$, of the following problem
\begin{equation}\label{eig_prob_in_M}
\rho[u,v]=\nu \mathcal{E}_{g,\cW,t}[u,v],
\quad in \quad (\cW,\mathcal{E}_{g,\cW,t}),
\end{equation}
where $\rho[\cdot,\cdot]$ is the form defined in Section \ref{stetment_of_the_problem}. Note that the norm, obtained by $\mathcal{E}_{g,\cW,t}[\cdot,\cdot]$ is equivalent to the standard norm in $H^1(M)$. Therefore we can equip $\cW$ with the norm $\mathcal{E}_{g,\cW,t}[\cdot,\cdot]$, and derive the new Hilbert space $(\cW,\mathcal{E}_{g,\cW,t})$. By the same arguments we do in Proposition \ref{rho_is_comp_con}, one can see that $\rho[\cdot,\cdot]$ is a completely continuous form in the Hilbert space $(\cW,\mathcal{E}_{g,\cW,t})$. Therefore, in the sense of Remark \ref{sense_of_eig_problem}, the eigenvalue problem \eqref{eig_prob_in_M} has discrete spectrum. We denote its non-zero eigenvalues by $\{-\lambda_{j}^{-}(\cW,t);\lambda_{j}^{+}(\cW,t)\}_{j=1}^{\infty}$, such that
\begin{equation*}
-\lambda_1^-(\cW,t)\leq-\lambda_2^-(\cW,t)\leq...<0<...\leq\lambda_2^+(\cW,t)\leq\lambda_1^+(\cW,t).
\end{equation*}

The following lemma allows us to localise the problem. 

\begin{lemma} 
	There exists a finite collection of open sets $\set{M_j}$ and functions $\{\Phi_j\}$ such that 
	\begin{enumerate}[(i)] 
		\item $(M_j, \Phi_j)$ is a coordinate patch,  
		\item $\partial M_j$ is piecewise smooth and Lipschitz,
		\item $M = \cup_{j=1}^K \overline{M_j}$ and $\mu_g(M \setminus \cup_{j=1}^K M_j) = 0$.  
	\end{enumerate} 
\end{lemma}
\begin{proof}
Every smooth manifold with boundary is smoothly triangulable (c.f. \cite{Cairns}), and so we take $\set{M_j}$ as the interior of the simplices in the triangulation.
The finiteness of the $\set{M_j}$ simply follows from compactness. 
It is easy to see that $\partial M_j$ is piecewise smooth and Lipschitz. 
The measure condition follows simply from the fact that each $\set{M_j}$ is a simplex.
\end{proof}

For each $k=1,...,K$, we define the forms
\begin{equation*}
	\begin{gathered}
	\mathcal{E}_k^D[u,v]:=(\nabla u,\nabla v )_{L^2(M_k, \ d\mu_g)}+t(u, v )_{L^2(M_k,\ d\mu_g)},
	 \quad \dom(\mathcal{E}_k^D)=H_0^1(M_k),\\
	\mathcal{E}_k^N[u,v]:=(\nabla u,\nabla v )_{L^2(M_k,\ d\mu_g)}+t(u, v )_{L^2(M_k,\ d\mu_g)},
	\quad \dom(\mathcal{E}_k^N)=H^1(M_k),\\
	\end{gathered}
\end{equation*}
and
\begin{equation*}
	\rho_k[u,v]:=\int_{M_k}\rho u \overline{v} \ d\mu_g
	\quad \dom(\rho_k)=H^1(M_k).
\end{equation*}
Note that the form $\rho_k[\cdot,\cdot]$ is a completely continuous symmetric form on $H^1(M_k)$, and its restriction on $H_0^1(M_k)$ is also a completely continuous symmetric form on $H_0^1(M_k)$. Therefore the eigenvalue problems
\begin{equation}\label{D_k_problem}
\rho_k [u,v]=\lambda\mathcal{E}_k^D[u,v]
\quad in \quad H_0^1(M_k),
\end{equation}
\begin{equation}\label{N_k_problem}
\rho_k [u,v]=\lambda\mathcal{E}_k^N[u,v]
\quad in \quad H^1(M_k),
\end{equation}
are well defined; see Remark \ref{sense_of_eig_problem}. We will investigate eigenvalues of the problems above by reducing them into Euclidean space. In order to do this let us introduce the following notions.

Given $T=T_{kj}dx^k\ \otimes dx^j$ with the matrix $(T_{kj})$ being invertible, we define $T_*=T^{kj}\frac{\partial}{\partial x_k}\otimes \frac{\partial}{\partial x_j}$ with $T_{kj}T^{kj} = \delta_k^j$, where $\delta_k^j$ is the Kronecker delta.

Let $\Omega_k:=\Phi^{k}(M_k)$ and $\Phi^k(\cdot,\cdot)$ be the pullback of the usual Euclidean inner product in $\Omega_k$. Fix a smooth metric, $h$, on $M$, as in Remark \ref{rmkrough}.  
Then, there exist $G$ and $H^k$ such that  
\begin{equation*}
\begin{gathered}
g(u,v)=h(Gu,v)\\
h(u,v)=\Phi^{k}(H^ku,v)
\end{gathered}
\end{equation*}
for $u$, $v\in L^2(T^\ast M_k, \ d\mu_g)$ (c.f. Proposition 10 in \cite{BRough}).  
Let $\theta^k:=\sqrt{\det H^k}$ and $\gamma:=\sqrt{\det G}$. We also set
\begin{equation*}
\begin{gathered}
\widetilde{G}^k:=G\circ\Phi_{k}^{-1},
\quad \widetilde{H}^k:=H^k\circ\Phi_{k}^{-1}\\
\tilde{\gamma}_k:=\gamma\circ\Phi_{k}^{-1},
\quad \tilde{\theta}_k:=\theta_k\circ\Phi_{k}^{-1},
\quad\tilde{\rho}_k:=\rho\circ \Phi_{k}^{-1}
\end{gathered}
\end{equation*}
Finally, we define 
\begin{equation*}
A_k:=\widetilde{H}_*^k\widetilde{G}_*^k\tilde{\theta}_k\tilde{\gamma}_k,
\qquad
B_k:=t\tilde{\theta}_k\tilde{\gamma}_k,
\qquad P_k:=\tilde{\rho}_k\tilde{\theta}_k\tilde{\gamma}_k.
\end{equation*}

Next we consider the following reduced\footnote{We refer to these forms as \emph{reduced} because the geometric setting has been reduced to a domain in $\R^n$.} forms  
\begin{equation*}
\begin{gathered}
\tilde{\mathcal{E}}_k^D[u,v]:=(A_k\nabla u,\nabla v )_{L^2(\Omega_k, d\cL)}+(B_k u, v )_{L^2(\Omega_k, d\cL)},
\quad \dom(\mathcal{E}_k^D)=H_0^1(\Omega_k),\\
\tilde{\mathcal{E}}_k^N[u,v]:=(A_k\nabla u,\nabla v )_{L^2(\Omega_k, d\cL))}+(B_k u, v )_{L^2(\Omega_k, d\cL)},
\quad \dom(\mathcal{E}_k^N)=H^1(\Omega_k),\\
\tilde{p}_k[u,v]:=\int_{\Omega_k}P_k u \overline{v} d \mathcal{L},
\end{gathered}
\end{equation*}
and the corresponding eigenvalue problems
\begin{equation}\label{D_k_problem_Omega}
\tilde{p}_k[u,v]=\lambda\tilde{\mathcal{E}}_k^D[u,v]
\quad in \quad H_0^1(\Omega_k),
\end{equation}
\begin{equation}\label{N_k_problem_Omega}
\tilde{p}_k[u,v]=\lambda \tilde{\mathcal{E}}_k^N[u,v]
\quad in \quad H^1(\Omega_k).
\end{equation}
According to Section \ref{Birman_Salomyak}, these problems have discrete spectrum. We denote their non-zero eigenvalues by $\left\{-\nu_{k,j}^-;\nu_{k,j}^+\right\}_{j=1}^{\infty}$ and $\left\{-\eta_{k,j}^-;\eta_{k,j}^+\right\}_{j=1}^{\infty}$, respectively. Next we prove that these are also eigenvalues of the problems \eqref{D_k_problem} and \eqref{N_k_problem} respectively.

\begin{lemma}\label{same_eigv}
	The eigenvalues of problems \eqref{D_k_problem} and \eqref{N_k_problem} coincide with eigenvalues of problems \eqref{D_k_problem_Omega} and \eqref{N_k_problem_Omega} respectively.
\end{lemma}
\begin{proof}
	Let both $u$ and $v$ belong to $H_0^1(\Omega_k)$ or $H^1(\Omega_k)$. Then, by the definitions of $g_*$, $h_*$, $G_*$, $\Phi^k _*$, $H^k _*$, $\gamma$, $\tilde H ^k _*$, $\tilde \theta_k$, $\tilde G_*$, $\tilde \gamma_k$, $\tilde u$, and $\tilde v$, 
	\begin{align*} 
	 (\nabla u, \nabla v)_{L^2(M_k, \ d\mu_g)} & = 
	\int_{M_k}g_*(\nabla u,\nabla v)\ d\mu_g \\ 
	&=\int_{M_k}h_*(G_*\nabla u,\nabla v)\ d\mu_g  \\
	&=\int_{M_k}h_*(G_*\gamma\nabla u,\nabla v)d\mu_h \\
	&=\int_{M_k}\Phi_*^k(H_*^k G_*\gamma\nabla u,\nabla v)d\mu_h \\
	&=\int_{\Omega_k}(\widetilde{H}_*^k\tilde{\theta}_k \widetilde{G}_*\tilde{\gamma}_k\nabla \tilde{u},\nabla  \tilde{v}) d\cL
	=\left(A_k\nabla \tilde{u},\nabla \tilde{v}\right)_{L^2(\Omega_k, d\cL)}.
	\end{align*}
	Similarly, one can check that
	\begin{equation*}
	\begin{gathered}
	t(u,v)_{L^2(M_k,\ d\mu_g)}=(t\tilde{\theta}_k\tilde{\gamma}_k \tilde{u},\tilde{v})_{L^2(\Omega_k,  d\cL)}=(B_k\tilde{u},\tilde{v})_{L^2(\Omega_k, d\cL)},\\
	\rho_k [u,v]=\int_{M_k}\rho u \overline{v} \ d\mu_g=\int_{\Omega_k}\tilde{\rho}_k\tilde{\theta}_k\tilde{\gamma}_k \tilde{u}\overline{\tilde{v}} d\cL=\tilde{p}_k[\tilde{u},\tilde{v}].
	\end{gathered}
	\end{equation*}
	Therefore, for $\lambda\in\mathbb{C}$ and $I=D$ or $I=N$, we obtain
	\begin{equation*}
		\rho_k [u,v]-\lambda\mathcal{E}_k^I[u,v]=\tilde{p}_k[\tilde{u},\tilde{v}]-\lambda\tilde{\mathcal{E}}_k^I[\tilde{u},\tilde{v}].\qedhere
	\end{equation*}
\end{proof}

By Lemma \ref{same_eigv}, problems \eqref{D_k_problem}, \eqref{N_k_problem} have the same eigenvalues as problems \eqref{D_k_problem_Omega}, \eqref{N_k_problem_Omega} respectively. From now on, by $\left\{-\nu_{k,j}^-;\nu_{k,j}^+\right\}_{j=1}^{\infty}$ and $\left\{-\eta_{k,j}^-;\eta_{k,j}^+\right\}_{j=1}^{\infty}$, we refer to the eigenvalues of the problems \eqref{D_k_problem} and \eqref{N_k_problem} respectively. The corresponding eigenfunctions we denote by $\left\{\phi_{k,j}^-;\phi_{k,j}^+\right\}_{j=1}^{\infty}$ and $\left\{\psi_{k,j}^-;\psi_{k,j}^+\right\}_{j=1}^{\infty}$ respectively.

\begin{prop}\label{main_loc_aux_prob}
	The eigenvalues of the problems \eqref{D_k_problem} and \eqref{N_k_problem} satisfy the following asymptotic formulas
	\begin{equation*}
	\lim_{j\rightarrow\infty}\nu_{k,j}^{\pm}j^{\frac{2}{n}}=\lim_{j\rightarrow\infty}\eta_{k,j}^{\pm}j^{\frac{2}{n}}=\left(\frac{\omega_n}{(2\pi)^{n}}\right)^{\frac{2}{n}}\left(\int_{M_k ^\pm }|\rho|^{\frac{n}{2}}\ d\mu_g\right)^{\frac{2}{n}},
	\end{equation*}
	where $M_k ^\pm := \{ x \in M_k : \pm \rho(x) \geq 0\}$.
\end{prop}
\begin{proof}
	Since $\left\{-\nu_{k,j}^-;\nu_{k,j}^+\right\}_{j=1}^{\infty}$ and $\left\{-\eta_{k,j}^-;\eta_{k,j}^+\right\}_{j=1}^{\infty}$ are also eigenvalues of problems \eqref{D_k_problem_Omega}, \eqref{N_k_problem_Omega} respectively, Theorem \ref{as_for_eig_prob_in_euq_B} implies
	\begin{equation}\label{main_loc_aux_prob1}
     \lim_{j\rightarrow\infty}\nu_{k,j}^{\pm}j^{\frac{2}{n}}=\lim_{j\rightarrow\infty}\eta_{k,j}^{\pm}j^{\frac{2}{n}}=\left(\frac{\omega_n}{(2\pi)^{n}}\right)^{\frac{2}{n}}\left(\int_{\Omega_k^{\pm}}\frac{|P_k(x)|^{\frac{n}{2}}}{\sqrt{\det A(x)}}dx\right)^{\frac{2}{n}}.
	\end{equation}
	where $\Omega_k^{\pm}:=\{x\in \Omega: \; \pm P_k(x)\geq0\}$ and $\omega_n$ is the volume of a unit ball in $\R^n$.   Since, for $x\in \Omega_k$, $\tilde{\theta}_k(x)>0$ and $\tilde{\gamma}_k(x)>0$, we obtain $\Omega_k^{\pm}=\{x\in \Omega: \; \pm \tilde{\rho}_k(x)\geq0\}$.
	Therefore
	\begin{align*}
	\int_{\Omega_k^{\pm}}\frac{|P_k(x)|^{\frac{n}{2}}}{\sqrt{\det A(x)}}dx&=
	\int_{\Omega_k^{\pm}}\frac{|\tilde{\rho}_k(x)\tilde{\theta}_k(x)\tilde{\gamma}_k(x)|^{\frac{n}{2}}}{\sqrt{\det \tilde{G}_*^k(x)\tilde{H}_*^k(x)}\left(\tilde{\theta}_k(x)\tilde{\gamma}_k(x)\right)^{\frac{n}{2}}}dx\\
	&=\int_{\Omega_k^{\pm}}\frac{|\tilde{\rho}_k(x)|^{\frac{n}{2}}}{\sqrt{\det \tilde{G}_*^k(x)\tilde{H}_*^k(x)}}dx\\
	&=\int_{\Omega_k^{\pm}}|\tilde{\rho}_k(x)|^{\frac{n}{2}}\tilde{\theta}_k(x)\tilde{\gamma}_k(x)dx.
	\end{align*}
	Since $\tilde{\rho}_k=\rho \circ \Phi_{k}^{-1}$, $\tilde{\rho}_k(x)>0$ iff $\rho\left(\Phi_{k}^{-1}(x)\right)>0$, and hence $\Phi_{k}^{-1}(\Omega_k^{\pm})=M_k^{\pm}$. Therefore the above equation gives
	\begin{align*}
	\int_{\Omega_k^{\pm}}\frac{|P_k(x)|^{\frac{n}{2}}}{\sqrt{\det A(x)}}dx=\int_{\Omega_k^{\pm}}|\tilde{\rho}_k(x)|^{\frac{n}{2}}\tilde{\theta}_k(x)\tilde{\gamma}_k(x)dx=\int_{M_k ^\pm }|\rho|^{\frac{n}{2}}\ d\mu_g.
	\end{align*}
	Hence \eqref{main_loc_aux_prob1} implies
	\begin{equation*}
		\lim_{j\rightarrow\infty}\nu_{k,j}^{\pm}j^{\frac{2}{n}}=\lim_{j\rightarrow\infty}\eta_{k,j}^{\pm}j^{\frac{2}{n}}=\left(\frac{\omega_n}{(2\pi)^{n}}\right)^{\frac{2}{n}}\left(\int_{M_k ^\pm }|\rho|^{\frac{n}{2}}\ d\mu_g\right)^{\frac{2}{n}}.
	\qedhere
	\end{equation*}
\end{proof}

Next we introduce the following notations. We set
\begin{equation*}
\begin{gathered}
\left\{-\nu_j^-;\nu_j^+\right\}_{j=1}^{\infty}:=\left\{-\nu_{k,j}^-;\nu_{k,j}^+\right\}_{k,j=1}^{\infty}\\
\left\{-\eta_j^-;\eta_j^+\right\}_{j=1}^{\infty}:=\left\{-\eta_{k,j}^-;\eta_{k,j}^+\right\}_{k,j=1}^{\infty}
\end{gathered}
\end{equation*}
such that
\begin{equation*}
\begin{gathered}
-\nu_1^-\leq -\nu_2^-\leq...<0<...\nu_2^+\leq \nu_1^+,\qquad
-\eta_1^-\leq -\eta_2^-\leq...<0<...\eta_2^+\leq \eta_1^+.
\end{gathered}
\end{equation*}
Therefore, every $\pm\nu_k^{\pm}$ is an eigenvalue of one of the problems \eqref{D_k_problem}, corresponding to some form $\mathcal{E}_{D,k}^{\pm}\in \{\mathcal{E}_{j}^{D}\}_{j=1}^{K}$ and domain $M_{D,k}^{\pm}\in \{M_j\}_{j=1}^{K}$. Similarly, every $\pm\eta_k^{\pm}$ is an eigenvalue of one of the problems \eqref{N_k_problem} corresponding  to some form $\mathcal{E}_{N,k}^{\pm}\in \{\mathcal{E}_{j}^{N}\}_{j=1}^{K}$ and domain $M_{N,k}^{\pm}\in \{M_j\}_{j=1}^{K}$. Let $\{\phi_k^{\pm}\}_{k=1}^{\infty}$ and $\{\psi_k^{\pm}\}_{k=1}^{\infty}$ be eigenfunctions corresponding to eigenvalues $\{\pm\nu_k^{\pm}\}_{k=1}^{\infty}$ and $\{\pm\eta_k^{\pm}\}_{k=1}^{\infty}$.

\begin{figure}[!htbp]\label{table}
\begin{tabular}{ | l | c | r | r | r| }
	
	\hline  & eig.val & eig.funct & domain & D. form \\ \hline
	 $k=1,...,K$ $j\in \mathbb{N}$ & $\pm\nu_{k,j}^{\pm}$ & $\phi_{k,j}^{\pm}$ & $M_k$ & $\mathcal{E}_k^D$ \\ \hline
	$k=1,...,K$ $j\in \mathbb{N}$ & $\pm\eta_{k,j}^{\pm}$ & $\psi_{k,j}^{\pm}$ & $M_k$ & $\mathcal{E}_k^N$ \\ \hline
	$j\in \mathbb{N}$ & $\pm\nu_{j}^{\pm}$ & $\phi_{j}^{\pm}$ & $M_{D,j}^{\pm}$ & $\mathcal{E}_{D,j}^{\pm}$ \\ \hline
	$j\in \mathbb{N}$ & $\pm\eta_{j}^{\pm}$ & $\psi_{j}^{\pm}$ & $M_{N,j}^{\pm}$ & $\mathcal{E}_{N,j}^{\pm}$ \\ \hline
	$j\in \mathbb{N}$ & $\pm\lambda_{j}^{\pm}(\cW,t)$ & $f_j^{\pm}$ & $M$ & $\mathcal{E}_{g,\cW,t}$ \\ \hline
	$j\in \mathbb{N}$ & $\pm\lambda_{j}^{\pm}(\cW)$ & $ $ & $M$ & $\mathcal{E}_{g,\cW}$ \\ \hline
\end{tabular} 
\caption{The eigenvalues and functions, together with their respective geometric domains and Dirichlet forms are organised in the above table.} 
\end{figure} 

\begin{prop}\label{est_for_D}
	We have the following estimate
	\begin{equation}\label{est_for_D1}
	\lambda_k^{\pm}(\cW,t)\geq \nu_k^{\pm},
	\quad k\in \mathbb{N}.
	\end{equation}
\end{prop}
\begin{proof}
	For $1\leq j\leq k$, we extend the eigenfunctions $\phi_j^{+}$ to $M\setminus M_{D,j}^{+}$ by zero. Note that
	this extension is in $H^1_0(M)$ since it is an eigenfunction for the Dirichlet problem. Since a system of $(k-1)$ linear equations with $k$ unknowns has a solution, we can find $\alpha_1,...,\alpha_k$ such that $f:=\sum_{j=1}^{k}\alpha_j\phi_j^{+}\in H^1(M)$ with $f \neq 0$ satisfies
	\begin{equation*}
	\mathcal{E}_{g,\cW,t}[f,f_j^+]=0.
	\end{equation*}
	Moreover, $f\in H^1_0(M)$ and so  we also have that $f \in \cW$ by our assumption on $\cW$.
	
	Recall that $(\cW,\mathcal{E}_{g,\cW,t})$ is a Hilbert space with scalar product $\mathcal{E}_{g,\cW,t}[\cdot,\cdot]$. Also recall that $\{\lambda_j^{\pm}(\cW,t)\}_{j=1}^{\infty}$ are non-zero eigenvalues of the completely continuous operator generated by the form $\rho [\cdot,\cdot]$. Let us denote this operator by $\mathcal{B}$, so that $\dom(\mathcal{B})=(\cW,\mathcal{E}_{g,\cW,t})$, and $\mathcal{E}_{g,\cW,t}[\mathcal{B}u,v]=\rho [u,v]$. Then  Theorem \ref{VP}\ref{VPIII}, with $A=\mathcal{B}$ and $\mathcal{H}=(\cW,\mathcal{E}_{g,\cW,t})$, implies
	\begin{equation*}
	\lambda_k^{+}(\cW,t)\mathcal{E}_{g,\cW,t}[f,f]\geq\mathcal{E}_{g,\cW,t}[\mathcal{B}f,f]=\rho [f,f]=\sum_{j,l=1}^{k}\rho[\alpha_j\phi_j^+,\alpha_l\phi_l^+].
	\end{equation*}
	Next let us note that $\rho[\phi_j^+,\phi_l^+]=0$ for $j\neq l$. Indeed, if their supports are disjoint then the claim is obviously true. If their supports intersect, then $M_{D,j}^{+}=M_{D,l}^{+}=M_i$ for some $i=1...K$. This means that $\phi_j^+$ and $\phi_l^{+}$ are distinct eigenfunctions of the $i$-th problem of \eqref{D_k_problem}. Therefore
	\begin{equation*}
	\nu_l^+\mathcal{E}_i^D[\rho\phi_j^+,\phi_l^+]=\rho_i[\phi_j^+,\phi_l^+]=\nu_j^+\mathcal{E}_i^D[\phi_j^+,\phi_l^+].
	\end{equation*}
	This is possible only if $\rho_i[\phi_j^+,\phi_l^+]=0$, and therefore $\rho[\phi_j^+,\phi_l^+]=0$.
	Hence the last estimate implies
	 \begin{equation}\label{est_for_D2}
	 \lambda_k^{+}(\cW,t)\mathcal{E}_{g,\cW,t}[f,f]\geq\sum_{j}^{k}\rho[\alpha_j\phi_j^+,\alpha_j\phi_j^+].
	 \end{equation}
	 On the other hand
	 \begin{align*}\label{est_for_D3}
	 \sum_{j=1}^{k}\rho[\alpha_j\phi_j^+,\alpha_j\phi_j^+]
	 &=\sum_{j=1}^{k} \nu^{+}_j\left((\nabla\alpha_j\phi_j^+,\nabla\alpha_j\phi_j^+)_{L^2(M,\ d\mu_g)}+t(\alpha_j\phi_j^+,\alpha_j\phi_j^+)_{L^2(M,\ d\mu_g)}\right)\\
	 &\geq\nu^{+}_k\sum_{j=1}^{k} \left((\nabla\alpha_j\phi_j^+,\nabla\alpha_j\phi_j^+)_{L^2(M,\ d\mu_g)}+t(\alpha_j\phi_j^+,\alpha_j\phi_j^+)_{L^2(M,\ d\mu_g)}\right)\\
	 &=\nu^{+}_k\mathcal{E}_{g,\cW,t}[f,f]
	 \end{align*}
	 since $\rho[\phi_j^+,\phi_l^+]=0$ for $l\neq j$. This, together with \eqref{est_for_D2}, implies \eqref{est_for_D1}. An analogous argument gives the result for the negative eigenvalues.
	 \end{proof}

\begin{prop}\label{est_for_N}
	We have the following estimate
	\begin{equation}\label{est_for_N1}
	\eta_k^{\pm}\geq \lambda_k^{\pm}(\cW,t),
	\end{equation}
	for sufficiently large $k\in \mathbb{N}$.
\end{prop}

\begin{proof}
	As in the previous proposition, we can find $f=\sum_{j=1}^{k}\beta_jf_j^+$ such that
	\begin{equation*}
	\mathcal{E}_{N,j}^{+}\left[\left.f\right|_{M_{N,j}^{+}},\psi_j ^+\right]=0,
	\quad j=1,...,k-1.
	\end{equation*}
	Let us fix $1\leq l\leq K$ to fix a chart $(\Phi_l, M_l)$. Next, we will prove the estimate
	\begin{equation}\label{est_for_N4}
	\eta_k^+\mathcal{E}_l^N[f\rest{M_l},f\rest{M_l}]\geq \rho \left[f\rest{M_l},f\rest{M_l}\right].
	\end{equation}
	Assume, for now, that there exists $m(l)\in \mathbb{N}$ such that
	\begin{equation} \label{existsml} 
	\eta_{l,m(l)+1}^{+}\leq\eta_k^{+}\leq\eta_{l,m(l)}^{+}.
	\end{equation}
	Recall that $\{\eta_{l,j}^{+}\}_{j=1}^{\infty}$ and $\{\psi_{l,j}^{+}\}_{j=1}^{\infty}$ are eigenvalues and eigenfunctions corresponding to the form $\mathcal{E}_l^{N}$ on the domain $M_l$. The last estimates imply $\{\eta_{l,j}^{+}\}_{j=1}^{m(l)}\subset \{\eta_{j}^{+}\}_{j=1}^{k}$. Therefore $\{\psi_{l,j}^{+}\}_{j=1}^{m(l)}\subset \{\psi_{j}^{+}\}_{j=1}^{k}$, and consequently, by from the construction of $f$, it follows
	\begin{equation}\label{est_for_N2}
	\mathcal{E}_l^{N}\left[\left.f\right|_{M_{l}},\psi_{l,j}^{+}\right]=0,
	\quad j=1,...,m(l).
	\end{equation}
	Since $f_j ^+\in \cW$, we see that $f\in \cW\subset H^1(M)$, and therefore $\left.f\right|_{M_l}\in H^1(M)=\dom(\mathcal{E}_l^{N})$. Moreover, by \eqref{est_for_N2}, we see that $\left.f\right|_{M_l}\perp \{\psi_{l,j}^{+}\}_{j=1}^{m(l)}$ in $\left(H^1(M_l), \mathcal{E}_l^{N}\right)$. Therefore, by Theorem \ref{VP}\ref{VPIII}, we obtain
	\begin{equation*}
	\eta_k^+\mathcal{E}_l^N[f\rest{M_l},f\rest{M_l}]\geq\eta_{l,m(l)+1}^+\mathcal{E}_l^N[f\rest{M_l},f\rest{M_l}]\geq \rho \left[f\rest{M_l},f\rest{M_l}\right].
	\end{equation*}
	Next, assume that there is no such $m(l)$ as in \eqref{existsml}. This is possible only if the eigenvalue problem \eqref{N_k_problem}, with number $l$, does not have positive eigenvalues. This means that the right hand side of \eqref{est_for_N4} is negative, so that \eqref{est_for_N4} still holds.
	
	Summing \eqref{est_for_N4} over $1\leq l\leq K$ gives
	\begin{equation}\label{est_for_N3}
	\eta_k^+\mathcal{E}_{g,\cW,t}[f,f]\geq \rho [f,f].
	\end{equation}
	Since
	\begin{equation*}
	\rho [f_j ^+,f_l ^+]=\lambda_{j}^+(\cW,t)\mathcal{E}_{g,\cW,t}[f_j ^+,f_l ^+],
	\end{equation*}
	we conclude that $\rho [f_j ^+,f_l ^+]=0$ for $j\neq l$. Therefore
	\begin{align*}
	\rho [f,f]=\sum_{j=1}^{k} \rho[\beta_jf_j ^+,\beta_jf_j ^+]
	&=\sum_{j=1}^{k}\lambda_{j}^+(\cW,t)\mathcal{E}_{g,\cW,t}[\beta_jf_j ^+,\beta_jf_j ^+]\\
	&\geq
	\lambda_{k}^+(\cW,t)\sum_{j=1}^{k}\mathcal{E}_{g,\cW,t}[\beta_jf_j ^+,\beta_jf_j ^+]\\
	&=\lambda_{k}^+(\cW,t)\mathcal{E}_{g,\cW,t}[f,f].
	\end{align*}
	Comparing this with \eqref{est_for_N3} we derive the statement. A similar argument proves the analogous result for the negative eigenvalues.
\end{proof}

Now we are ready to prove the main theorem of this subsection.

\begin{theorem}\label{main_aux_problem}
	The eigenvalues of problem \eqref{eig_prob_in_M} satisfy the following asymptotic formula
	\begin{equation*}
	\lim_{k\rightarrow\infty}\lambda_k^{\pm}(\cW,t)k^{\frac{2}{n}}=\left(\frac{\omega_n}{(2\pi)^{n}}\right)^{\frac{2}{n}}\left(\int_{M^{\pm}}|\rho|^{\frac{n}{2}}\ d\mu_g\right)^{\frac{2}{n}},
	\end{equation*}
	where $M^{\pm}:=\{x\in M: \pm\rho(x)>0\}$.
\end{theorem}
\begin{proof}
	We note that for each $k$, by Proposition \ref{main_loc_aux_prob}, the counting functions satisfy 
	$$\lim_{\lambda \to 0} \lambda^{n/2} \# \{ \eta_{k,j} ^\pm \geq \lambda \} = \frac{\omega_n}{(2\pi)^n} \int_{M_k ^\pm} |\rho|^{\frac n 2} \ d\mu_g,$$
	$$\lim_{\lambda \to 0} \lambda^{n/2} \# \{ \nu_{k,j} ^\pm \geq \lambda \} = \frac{\omega_n}{(2\pi)^n} \int_{M_k ^\pm} |\rho|^{\frac n 2} \ d\mu_g.$$
	Consequently, 
	$$\lim_{\lambda \to 0} \sum_{k=1} ^K \lambda^{n/2}  \# \{ \eta_{k,j} ^\pm \geq \lambda \} = \frac{\omega_n}{(2\pi)^n}  \sum_{k=1} ^K \int_{M_k ^\pm} |\rho|^{\frac n 2} \ d\mu_g =  \frac{\omega_n}{(2\pi)^n} \int_{M^\pm} |\rho|^{\frac n 2} \ d\mu_g,$$
	and similarly 
	$$\lim_{\lambda \to 0} \sum_{k=1} ^K \lambda^{n/2}  \# \{ \nu_{k,j} ^\pm \geq \lambda \} =\frac{\omega_n}{(2\pi)^n} \int_{M^\pm} |\rho|^{\frac n 2} \ d\mu_g.$$
	Let $N^{\pm}(\lambda,\mathcal{E}_{g,\cW, t})$, $N^{\pm}(\lambda,\mathcal{E}_k^{D})$, and $N^{\pm}(\lambda,\mathcal{E}_k^{N})$ be the counting functions of the eigenvalues of problems \eqref{eig_prob_in_M}, \eqref{D_k_problem}, and \eqref{N_k_problem} respectively.  
	By their very definitions, 
	$$\sum_{k=1} ^K \# \{ \eta_{k,j} ^\pm \geq \lambda \} = \sum_{k=1} ^K N^{\pm}(\lambda,\mathcal{E}_k^{N}),$$
	and similarly, 
	$$\sum_{k=1} ^K \# \{ \nu_{k,j} ^\pm \geq \lambda \} = \sum_{k=1} ^K N^{\pm}(\lambda,\mathcal{E}_k^{D}).$$
	We therefore have 
	$$\lim_{\lambda \to 0} \lambda^{\frac n 2}  \sum_{k=1} ^K N^{\pm}(\lambda,\mathcal{E}_k^{N}) = \lim_{\lambda \to 0} \lambda^{\frac n 2}  \sum_{k=1} ^K N^{\pm}(\lambda,\mathcal{E}_k^{D}) = \frac{\omega_n}{(2\pi)^n} \int_{M^\pm} |\rho|^{\frac n 2} \ d\mu_g.$$

	By Propositions \eqref{est_for_D} and \eqref{est_for_N}
	\begin{equation*}
	\sum_{k=1}^{K}\lambda^{\frac n 2} N^{\pm}(\lambda,\mathcal{E}_k^{D})\leq \lambda^{ \frac n 2} N^{\pm}(\lambda,\mathcal{E}_{g,\cW, t})\leq
	\sum_{k=1}^{K} \lambda^{\frac n 2} N^{\pm}(\lambda,\mathcal{E}_k^{N}).
	\end{equation*}
	Thus, we obtain 
	$$\lim_{\lambda \to 0} \lambda^{\frac n 2}  N^{\pm}(\lambda,\mathcal{E}_{g,\cW, t}) = \frac{\omega_n}{(2\pi)^n} \int_{M^\pm} |\rho|^{\frac n 2} \ d\mu_g.$$
	The statement of the theorem is an immediate consequence.  
\end{proof}

\subsection{Eigenvalue asymptotics for the weighted Laplacian on a rough Riemannian manifold}
In this subsection we will prove our main result. We start with the following lemma which allows us to derive the asymptotics of $\lambda_{k}(\cW)$ from those of $\lambda_{k}(\cW,t)$. We note that this lemma is an adaptation of \cite[Lemma 2.1]{BS1970}.

\begin{lemma}
	We have the following estimates
	\begin{equation}\label{est_main_ev}
	\lambda_{k+\tau}^{\pm}(\cW,t)\leq \lambda_{k}^{\pm}(\cW)\leq (1-t)^{-1}\lambda_{k}^{\pm}(\cW,Ct),
	\qquad 0<t<1
	\end{equation}
	for some $C>0$ independent of $t>0$. (Recall that $\tau=\dim Z(\rho)^{\perp}$, and $\tau\leq 1$).  
\end{lemma}

\begin{proof}
	By Proposition \ref{equivalence_of_norms}, there exists $C>0$ such that $\mathcal{E}_{g,\cW}[u,u]\geq C(u,u)_{L^2(M,\ d\mu_g)}$ for all $u\in Z(\rho)$.
	Therefore  
	\begin{equation*}
	\mathcal{E}_{g,\cW}[u,u]\geq(1-t)\left(\mathcal{E}_{g,\cW}[u,u]+Ct(u,u)_{L^2(M,\ d\mu_g)}\right),
	\qquad 0<t<1.
	\end{equation*}
	Therefore, by applying Theorem \ref{VP}\ref{VPII} with $\mathcal{H}=\left(Z(\rho),\mathcal{E}_{g,\cW}\right)$ and $A$ being the operator generated by the form $\rho [\cdot,\cdot]$, we conclude 
	\begin{align*}
	\lambda_{k}^{+}(\cW)=&\max_{L\subset Z(p),\dim L=k}\min_{u\in L\setminus\{0\}}\frac{\rho [u,u]}{\mathcal{E}_{g,\cW}[u,u]}\\
	&\leq\max_{L\subset Z(p),\dim L=k}\min_{u\in L\setminus\{0\}}\frac{\rho [u,u]}{(1-t)\left(\mathcal{E}_{g,\cW}[u,u]+Ct(u,u)_{L^2(M,\ d\mu_g)}\right)}
	\\
	&\leq\max_{L\subset \cW,\dim L=k}\min_{u\in L\setminus\{0\}}\frac{\rho [u,u]}{(1-t)\left(\mathcal{E}_{g,\cW}[u,u]+Ct(u,u)_{L^2(M,\ d\mu_g)}\right)}
	\\
	&=(1-t)^{-1}\lambda^{+}_k(\cW, Ct).
	\end{align*}
	In the last equation we again applied Theorem \ref{VP}\ref{VPII}, but with $\mathcal{H}=\left(\cW,\mathcal{E}_{g,\cW,Ct}\right)$ and $A$ being the operator generated by the form $\rho [\cdot,\cdot]$.	This proves the second inequality. 
	
	In the same way, but using  Theorem \ref{VP}\ref{VPI}, we derive
	\begin{align*}
	\lambda_{k+\tau}^{+}(\cW,t)=&\min_{L\subset \cW, \dim L^{\perp}=k+\tau-1}\max_{u\in L\setminus\{0\}}\frac{\rho [u,u]}{\left(\mathcal{E}_{g,\cW}[u,u]+t(u,u)_{L^2(M,\ d\mu_g)}\right)}\\
	&\leq \min_{L\subset \cW, \dim L^{\perp}=k+\tau-1,(Z(p))^{\perp}\subset L^{\perp}}\max_{u\in L\setminus\{0\}}\frac{\rho [u,u]}{\left(\mathcal{E}_{g,\cW}[u,u]+t(u,u)_{L^2(M,\ d\mu_g)}\right)}\\
	&\leq \min_{L\subset Z(\rho), \dim L^{\perp}=k-1}\max_{u\in L\setminus\{0\}}\frac{\rho [u,u]}{\left(\mathcal{E}_{g,\cW}[u,u]+t(u,u)_{L^2(M,\ d\mu_g)}\right)}\\
	&=\lambda_{k}^{+}(\cW).
	\end{align*}
	This proves the first estimate. An analogous argument shows the same result for the negative eigenvalues.
\end{proof}

We are now poised to prove the main theorem.  The statements concerning the discreteness of the spectrum have already been proven, so it only remains to demonstrate

\begin{theorem}[\textbf{Weyl's law for a weighted Laplacian with an admissible boundary condition}]
We have the following asymptotic formula 
\begin{equation*}
\lim_{k\rightarrow\infty}\lambda_{k}^{\pm}(\cW)k^{\frac{2}{n}}=\left(\frac{\omega_n}{(2\pi)^{n}}\right)^{\frac{2}{n}}\left(\int_{M^\pm}|\rho(x)|^{\frac{n}{2}}\ d\mu_g\right)^{\frac{2}{n}}=\left(\frac{\omega_n}{(2\pi)^{n}}\right)^{\frac{2}{n}}\|\rho\|_{L^{\frac{n}{2}}(M^\pm, \ d\mu_g)}.
\end{equation*}
\end{theorem}
\begin{proof}
	Let us multiply \eqref{est_main_ev} by $k^{\frac{2}{n}}$ and take the limit as $k\rightarrow\infty$, 
	\begin{equation*}
	\lim_{k\rightarrow\infty}\lambda_{k+\tau}^{\pm}(\cW,t)k^{\frac{2}{n}}\leq \lim_{k\rightarrow\infty}\lambda_{k}^{\pm}(\cW)k^{\frac{2}{n}}\leq (1-t)^{-1}\lim_{k\rightarrow\infty}\lambda_{k}^{\pm}(\cW,Ct)k^{\frac{2}{n}}.
	\end{equation*}
	We have already demonstrated that 
	\begin{equation*}
	\lim_{k\rightarrow\infty}\lambda_{k}^{\pm}(\cW,Ct)k^{\frac{2}{n}}=\left(\frac{\omega_n}{(2\pi)^{n}}\right)^{\frac{2}{n}}\left(\int_{M^\pm}|\rho(x)|^{\frac{n}{2}}\ d\mu_g\right)^{\frac{2}{n}}.
	\end{equation*}
	Similarly, 
	\begin{equation*}
	\lim_{k\rightarrow\infty}\lambda_{k+\tau}^{\pm}(\cW,t)(k+\tau)^{\frac{2}{n}}=\left(\frac{\omega_n}{(2\pi)^{n}}\right)^{\frac{2}{n}}\left(\int_{M^\pm}|\rho(x)|^{\frac{n}{2}}\ d\mu_g\right)^{\frac{2}{n}}.
	\end{equation*}
	Since $\tau \in \{0, 1\}$, we clearly have 
	$$\lim_{k \to \infty} \frac{k^{\frac 2 n}}{(k+\tau)^{\frac 2 n}} = 1.$$
	Therefore, 
	\begin{equation*}
	\lim_{k\rightarrow\infty}\lambda_{k+\tau}^{\pm}(\cW,t)(k)^{\frac{2}{n}}=\left(\frac{\omega_n}{(2\pi)^{n}}\right)^{\frac{2}{n}}\left(\int_{M^\pm}|\rho(x)|^{\frac{n}{2}}\ d\mu_g\right)^{\frac{2}{n}}.
	\end{equation*}
	
	Thus, we  derive
	\begin{multline*}
	\left(\frac{\omega_n}{(2\pi)^{n}}\right)^{\frac{2}{n}}\left(\int_{M^\pm}|\rho(x)|^{\frac{n}{2}}\ d\mu_g\right)^{\frac{2}{n}} \\ 
		\leq \lim_{k\rightarrow\infty}\lambda_{k}^{\pm}(\cW)k^{\frac{2}{k}}\leq (1-t)^{-1}\left(\frac{\omega_n}{(2\pi)^{n}}\right)^{\frac{2}{n}}\left(\int_{M^\pm}|\rho(x)|^{\frac{n}{2}}\ d\mu_g\right)^{\frac{2}{n}}.
	\end{multline*}
	Finally, by taking $t\rightarrow0$, we obtain the statement.  
\end{proof}


\section{Concluding Remarks} \label{Sec:CR}
In this paper,  we considered Laplacians induced by rough metrics $g$ and weighted eigenvalue equations involving a weight function, $\rho$, which need not have constant sign.  The manifold in question was also permitted to have boundary, and we were able to consider a large class of admissible boundary conditions including mixed boundary conditions.  There are a number of directions that further research for such problems could take.

An immediate and interesting question is to determine estimates for the remainder term in Weyl's law.  Since this contains curvature information in the smooth case, it would be interesting to understand what this reveals about the structure of rough metrics, and perhaps this would allow us to extract a weak notion of curvature or curvature bounds for such objects.

Beyond this question, we can consider this problem in more general settings.  One direction would be to consider $(V,h) \to M$, a Hermitian vector bundle with metric $h$ over $M$ equipped with a measure $\mu$.
Fixing a connection $\nabla$ and a closed subspace $\cW \subset H^1(V,h)$ with $H_0^1(V,g) \subset \cW$,  we could consider the eigenvalue problem for the divergence form equation $b \nabla^\ast_\cW  B \nabla_\cW$, where  $\dom(\nabla_{\cW}) = \cW$, $b$ is a measurable function bounded below, and $B$ is an elliptic, bounded, measurable endomorphism over $V$.  The complication of this analysis is the fact that we can no longer localise the problem by pulling it into $\R^n$, and we would have to devise a method by which we only use the trivialisations available to us from the bundle structure.
If the measure $\mu$ were to not be induced by a rough metric, then we would also need to understand which classes of measures would be appropriate.

A bundle in which we have the commutativity of the pullbacks with a differential operator are the differential forms $\Omega M$, where we would be forced to consider the exterior derivative $d$ instead of a connection $\nabla$. 
Fixing a rough metric $g$, we would obtain adjoints $d_g^\ast$, and the Hodge-Dirac operator, $D_H = d + d_g^\ast$, would be an operator of interest.  This analysis would be plausible on a manifold without boundary to obtain spectral asymptotics for the Hodge-Laplacian  $\Delta_{H} = D^2$, using similar methods to those that we used here. 
However, the presence of boundary would complicate matters, since it is known that even in the setting when the metric is smooth, the operator
$$ \Delta_{H} = D_{\max}^\ast D_{\max}$$
where $D_{\max} = (D_c)^\ast$ with domain $\dom(D_c) = C_c^\infty(\Omega M)$, admits an infinite dimensional kernel (c.f. Proposition 3.5 in \cite{Albin}).\footnote{This result was communicated privately to us by Matthias Ludewig, Batu Güneysu and Francesco Bei.}
The analysis would therefore require an understanding of the boundary conditions we impose on the boundary.  
We intend to investigate these problems in forthcoming work   and regard the present paper as a solid foundation upon which to initiate a more general study.  

\bibliographystyle{alpha}
\bibliography{roughweyl}

\setlength{\parskip}{0pt}

\end{document}